%% file: A_Boolean_Algebraic_Approach_to_Semiproper_Iterations.tex
\documentclass[11pt]{article}

\usepackage{amsmath,amsthm}
\usepackage{amssymb,amsfonts}
\usepackage{hyperref}
\usepackage{tikz}

\usepackage{stmaryrd}
\DeclareSymbolFont{bbold}{U}{bbold}{m}{n}
\DeclareSymbolFontAlphabet{\mathbbm}{bbold}

\theoremstyle{plain}
	\newtheorem{theorem}{Theorem}[section]
	\newtheorem{proposition}[theorem]{Proposition}
	\newtheorem{lemma}[theorem]{Lemma}
	\newtheorem{corollary}[theorem]{Corollary}

\theoremstyle{definition}
	\newtheorem{definition}[theorem]{Definition}
	\newtheorem{notation}[theorem]{Notation}
	\newtheorem{fact}[theorem]{Fact}

\theoremstyle{remark}
	\newtheorem{remark}[theorem]{Remark}

	\newtheorem{example}[theorem]{Example}

\newcommand{\ZFC}{\ensuremath{\text{{\sf ZFC}}}}

\DeclareMathOperator{\dom}{dom}
\DeclareMathOperator{\ran}{ran}

\DeclareMathOperator{\coker}{coker}

\DeclareMathOperator{\crit}{crit}

\DeclareMathOperator{\cf}{cf}

\DeclareMathOperator{\supp}{supp}

\DeclareMathOperator{\val}{val}

\newcommand{\FA}{\ensuremath{\text{{\sf FA}}}}

\newcommand{\SPFA}{\ensuremath{\text{{\sf SPFA}}}}
\newcommand{\Nm}{\ensuremath{\text{{\sf Nm}}}}
\newcommand{\SP}{\ensuremath{\text{{\sf SP}}}}

\newcommand{\MM}{\ensuremath{\text{{\sf MM}}}}

\newcommand{\non}{\ensuremath{\lnot}}

\newcommand{\NS}{\ensuremath{\text{{\sf NS}}}}
\newcommand{\CF}{\ensuremath{\text{{\sf CF}}}}

\newcommand{\RR}{\mathbb{D}}
\newcommand{\QQ}{\mathbb{C}}
\newcommand{\BB}{\mathbb{B}}
\newcommand{\DD}{\mathbb{D}}

\newcommand{\PP}{\mathbb{P}}

\newcommand{\A}{{\sf A}}
\newcommand{\PD}{{\sf PD}}
\newcommand{\RO}{{\sf RO}}
\newcommand{\RCS}{{\sf RCS}}
\newcommand{\sg}{sg}
\newcommand{\gen}{gen}
\newcommand{\FFF}{\mathcal{F}}

\DeclareMathOperator{\Coll}{Coll}

\newcommand{\pp}{\mathcal{P}}
\newcommand{\res}{\upharpoonright}
\newcommand{\cp}[1]{\left( #1 \right)}
\newcommand{\qp}[1]{\left[ #1 \right]}
\newcommand{\Qp}[1]{\left\llbracket #1 \right\rrbracket}
\newcommand{\ap}[1]{\langle #1 \rangle}
\newcommand{\bp}[1]{\left\lbrace #1 \right\rbrace}
\newcommand{\vp}[1]{\left\lvert #1 \right\rvert}
\newcommand{\0}{\mathbbm 0}
\newcommand{\1}{\mathbbm 1}
\newcommand{\2}{\mathbbm 2}
\newcommand{\ON}{\mathrm{\mathbf{ON}}}
\newcommand{\psup}{\tilde{\bigvee}}

\title{A boolean algebraic approach to semiproper iterations}
\author{M. Viale, G. Audrito, S. Steila}
\date{}

\begin{document}
	\maketitle

	\input{0_Introduction}

	\tableofcontents
	\input{1_Posets}

	\input{2_Embeddings}

	\input{3_Iterations}

	\input{4_Quotients}
	\input{5_Examples}
	\input{6_Semiproperness}
	\input{7_SPIterations}
	\input{8_MM}

	\appendix	
	
	\input{9_Stationary-sets}

	\bibliographystyle{amsplain}
	\bibliography{9_Biblio}
\end{document}

%% file: 0_Introduction.tex
\section*{Introduction}

These notes present a compact and self contained development of the theory of
iterated forcing with a focus on semiproperness and revised countable support iterations.
We shall pursue the approach to iterated forcing devised by Donder and Fuchs 
in~\cite{FUCHS}, thus we shall present iterated forcing by means 
of directed system of complete and injective homomorphisms of complete boolean algebras.
A guiding idea that drives this work is 
that for many purposes, especially when dealing with problems of a 
methamatematical nature, the use of boolean valued
models is more convenient.
A partial order and its boolean completion can produce exactly the same consistency results,
however:
\begin{itemize}
\item
In a specific consistency proof the forcing notion 
we have in mind in order to obtain the desired result is given by a partial order
and passing to its boolean completion may obscure our intuition on the nature of the problem and the combinatorial properties we wish our partial order to have.
\item
When the 
problem aims to find general properties of 
forcings which are shared by a wide class of partial orders, we believe 
that focusing on complete boolean algebras gives a more efficient way to handle the problem.
This is the case for at least two reasons: on the one hand
there are less complete boolean algebras to deal with than partial orders, 
thus we have to handle potentially less objects, on the other hand we have a rich algebraic 
theory for complete boolean algebras and the use of algebraic properties may greatly 
simplify our calculations.
\end{itemize}
We believe that this second case applies when our aim is to develop a general
theory of iterated forcing and these notes are guided by this convinction.

The first five sections of these notes contain a detailed presentation of the algebraic properties 
of complete homomorphisms between atomless complete boolean algebras and basic facts
on limits of directed systems of complete homomorphisms between boolean algebras.
Sections 6 and 7 introduce a boolean algebraic definition of semiproperness and 
develop the core results on semiproperness and on revised countable 
support iterations of semiproper forcings basing them on this algebraic characterization 
of semiproperness.
Section 8 contains a proof of the most celebrated application of semiproperness,
i.e. Foreman, Magidor, and Shelah's proof of the consistency of $\MM$~\cite{FOREMAGISHELAH}. 
Since a crucial role in our analysis of semiproperness
is played by generalized stationary sets, 
we enclude an appendix containing all relevant facts
about generalized stationarity that were employed in these notes.

The results we present are well established parts of the current development of set theory,
however the proofs are novel and in some cases we believe to cover a gap
in the literature, especially in light of the fact that, 
up to date, there is (in our eyes)
no neat self contained presentation of the preservation theorems for semiproperness under revised 
countable support iterations.
These notes take a great care to present all basic results and to give detailed proofs
(provided these are not covered in a systematic way elsewhere), for this reason we believe they are
of interest to any scholar who is acquainted with forcing and the basics of boolean valued models 
and aims to learn the standard results on proper and semiproper iterations.
While the focus is on semiproper iterations we believe there will be no problem 
to rearrange these techniques in order to cover also
the cases of proper or ccc iterations.

The paper is organized as follows:
\begin{itemize}
\item
Section~\ref{sec:background} 
contains in the first part basic material on the relation between partial orders, 
their boolean completions, the Stone spaces associated to their boolean completions. 
In the second part of the section we give
a sketchy presentation of the basic properties of boolean valued models. We assume the reader is 
acquainted with these results. For the part on partial orders, boolean completions and Stone spaces,
a source of inspiration can be chapter 2 of
Kunen's book~\cite{KUNEN}, for the part on Boolean valued models we refer the reader to 
Bell's book~\cite{BELL}, to
Jech's chapter on forcing~\cite{JECH}, to Hamkins and Seabold's paper on 
Boolean ultrapowers~\cite{HAMSEA}, or to Audrito's master thesis~\cite{AUDRITO}.

\item
In the first part of section~\ref{sec:regemb} we introduce regular homomorphisms between 
atomless complete boolean algebras (i.e. injective complete homomorphisms)
and their associated retractions.

$\pi:\QQ\to\BB$ is the retraction associated to a regular homomorphism
$i:\BB\to\QQ$ if
\[
\pi(q)=\bigwedge_{\BB}\{b:i(b)\geq q\}.
\]
The key feature of these retractions is the identity
\[
\pi(i(b)\wedge q)=b\wedge \pi(q)
\]
for all $b\in\BB$ and $q\in\QQ$.
This algebraic identity will be the cornerstone in our analysis of iterated forcing.
We prove in details this and other identities and some other facts: for example that any
complete homomorphism $i:\BB\to\QQ$ grants that whenever we add a $V$-generic filter $G$
for $\QQ$ then $i^{-1}[G]$ is a $V$-generic filter for $\BB$, i.e. 
in the context of boolean valued models complete homomorphisms,
play the role complete embeddings between posets have in the context of ordinary forcing.

In the second part of this section we give a proof that complete homomorphisms $i:\BB\to\QQ$ 
induce $\Delta_1$-preserving embeddings $\hat{\imath}:V^{\BB}\to V^{\QQ}$ 
on the respective boolean valued models.
\item
In section~\ref{sec:iterations} we present iterated forcing in the setting of complete boolean algebras.
In this section occurs the first great simplification (due to Donder and Fuchs) that our presentation of 
iterated forcing allows, which is the definition of revised countable support iterations.
For this reason we wish to spend some more words on the matters treated in this section.

We focus our presentation limiting our attention to iteration systems of regular homomorphisms,
which are the exact counterpart in our setting of the standard notion of iteration for posets.
A complete iteration system 
\[
\FFF=\{i_{\alpha \beta}:\BB_\alpha\to \BB_\beta:\alpha \leq \beta < \lambda\}
\]
is a commuting family of regular embeddings, a branch in $T(\FFF)$ is a function
$f:\lambda\to V$ such that for all $\alpha\leq\beta<\lambda$
$f(\alpha)$ is the retraction of $f(\beta)$ by means of the retraction associated to 
$i_{\alpha\beta}$.
A branch $f$ is eventually constant if there is some $\alpha$ such that
$f(\beta)=i_{\alpha\beta}\circ f(\alpha)$ for every $\beta\geq\alpha$.
There is a standard order on $T(\FFF)$ given by the pointwise comparison of branches.
With respect to the standard presentations
of iterated forcing,
$C(\FFF)$, the set of constant branches, corresponds to the direct limit, 
$T(\FFF)$ corresponds to the full limit.

The revised countable support of $\FFF$ consists of those branches $f\in T(\FFF)$ 
with the property that
\begin{itemize}
\item
either for some $\alpha<\lambda$ $f(\alpha)$ forces with respect to $\BB_\alpha$ that\\
$\cf(\lambda)=\omega$,
\item
or $f$ is eventually constant.
\end{itemize} 
We invite the reader to compare this definition of revised countable support iterations
with the original one\footnote{We do not dare to check whether the two 
notions of revised countable support limit are equivalent, however
they serve the same purpose (i.e producing iterations of semiproper posets that do not collapse 
$\omega_1$) and it is transparent that Donder and Fuchs'definition is more
manageable than Shelah's one.} of Shelah in Chapter $X$ of his book~\cite{SHEPRO}.

In the second part of the section
we study basic properties of complete iteration systems, in particular we set up 
sufficient conditions to establish when the direct limit of an iteration system of length $\lambda$
is ${<}\lambda$-cc (this corresponds to the well known 
result of Baumgartner on direct limits of ${<}\lambda$-cc forcings),
and when $T(\FFF)$ and $C(\FFF)$ do overlap.

\item
In the first part of
section~\ref{sec:twostepitquot} we analyze in more details the relation existing between
$V^{\BB}$ and $V^{\QQ}$ in case $i:\BB\to\QQ$ is  a complete homomorphism.
First we show that if $G$ is $V$-generic for $\BB$ and $K$ is the dual of $G$, $\QQ/i[K]$
is a complete boolean algebra in $V[G]$, and also that whenever
	\[
	\begin{tikzpicture}[xscale=1.5,yscale=-1.2]
		\node (B) at (0, 0) {$\BB$};
		\node (Q0) at (1, 0) {$\QQ_0$};
		\node (Q1) at (1, 1) {$\QQ_1$};
		\path (B) edge [->]node [auto] {$\scriptstyle{i_0}$} (Q0);
		\path (B) edge [->]node [auto,swap] {$\scriptstyle{i_1}$} (Q1);
		\path (Q0) edge [->]node [auto] {$\scriptstyle{j}$} (Q1);
	\end{tikzpicture}
	\]
and $K$ is the dual of a $V$-generic filter $G$ for $\BB$, then the map 
defined by $j/_K([c]_{i_0[K]})=[j(c)]_{i_1[K]}$ is a complete
homomorphism in $V[G]$.

In the second part of this section 
we show that whenever $\dot{\QQ}\in V^{\BB}$ is a $\BB$-name for a complete boolean algebra
there is a complete boolean algebra $\BB*\dot{\QQ}\in V$ and
an $i:\BB\to \BB*\dot{\QQ}$ such that whenever $G$ is $V$-generic for $\BB$ and $J$ is its dual
$\BB*\dot{\QQ}/i[J]$ is isomorphic to $\dot{\QQ}_G$ in $V[G]$.

Finally in the last part of the section, 
with these results at our disposal, we check that $\FFF/G$ is a complete iteration system in
$V[G]$ whenever
$\FFF$ is an iteration system in $V$ and $G$ is $V$-generic for some $\BB$ in $\FFF$.

\item
Section~\ref{sec:examples} 
gather a family of examples which should clear up many misleading points on the properties of
complete iteration systems.

\item
Sections~\ref{sec:semiproper} and~\ref{sec:semiproperiter} 
contain the bulk of our results on semiproperness. In section~\ref{sec:semiproper}
we introduce the second key simplification in our treatment of semiproper iterations.
We introduce an algebraic definition of semiproperness which due to its relevance we wish to 
anticipate here:

given a countable model $M\prec H_\theta$ and a complete boolean algebra $\BB\in M$,
\[
\sg(\BB,M)=\bigwedge_{\BB}\bp{\bigvee_{\BB} (X\cap M): X\in M\text{ is a predense subset of $\BB$ of size $\aleph_1$}}.
\]

$\BB$ is semiproper if for a club of countable $M\prec H_\theta$ and for all
$b\in \BB^+\cap M$
\[
b\wedge \sg(\BB,M) > \0_{\BB}.
\]

We also introduce the key notion to analyze semiproper iterations:

$i:\BB\to\QQ$ is a semiproper regular homomorphism if $\BB$ is semiproper and for a club of countable 
$M\prec H_\theta$ and for all $q\in \QQ^+\cap M$ we have that
\[
\pi(q)\wedge\sg(\BB,M)=
\pi(q\wedge \sg(\QQ,M)) > \0_{\BB}.
\]

Next we show that a forcing notion $P$ is semiproper according to Shelah's definition iff
its boolean completion is semiproper according to our definition.

We conclude the section giving a simple topological characterization of properness and semiproperness. 
\item
Section~\ref{sec:semiproperiter} 
is devoted to the analysis of two-step iterations of semiproper posets
and to the proof of the preservation of semiproperness through 
revised countable support limits.

First, along the same lines of what was done in section~\ref{sec:twostepitquot}, 
we prove that two-step and 
three-step iterations
of semiproper forcings behave as expected. In particular
we show that
whenever $\BB$ is a semiproper  complete boolean algebra
and $\dot{\QQ}\in V^{\BB}$ is a name for a semiproper complete boolean algebra
then the natural regular homomorphism $i:\BB\to \BB*\dot{\QQ}$ is also 
semiproper\footnote{We want to remark that $i:\BB\to\QQ$ can be a semiproper 
regular homomorphism even if
for some $G$ $V$-generic for $\BB$, 
$\QQ/i[J]$ is not semiproper in $V[G]$ (where $J$ is the dual of $G$).}.

Next we concentrate on the proof of the preservation of semiproperness through limit stages.
The proof splits in three cases according to the cofinality of the length of the iteration system 
($\aleph_0$, $\aleph_1$, bigger than $\aleph_1$) and mimicks in this new setting the original proof
of Shelah of these results. 
\item
Section~\ref{sec:consMM} 
gives a proof of the consistency of the forcing axiom $\MM$ relative to the existence
of a supercompact cardinal by means of a semiproper iteration.
\item
The appendix~\ref{sec:statsets} contains a detailed exposition of generalized stationary sets. 
The basic properties of these sets are needed to develop all properties of semiproper forcings
and thus are required in order to follow the content of sections~\ref{sec:semiproper},~\ref{sec:semiproperiter}, and~\ref{sec:consMM}.

\end{itemize}

These notes are the outcome of a Ph.D. course the first author gave on these matters
in the spring and the fall of 2013.
The basic ideas that guided the course comes from the observation that a full account of
Donder and Fuchs approach
to semiproper iterations is not available in a published form and the unique draft
of their results is the rather sketchy preprint on the ArXiv~\cite{FUCHS}. 
Moreover the available
drafts of their results do not push to their 
extreme consequences the power given by the algebraic apparatus 
provided by the theory of complete boolean algebra.
Donder and Fuchs limit themselves to use 
this algebraic apparatus to simplify (dramatically) the definition of revised countable
support limit.
No attempt is done by them to use 
this algebraic apparatus to simplify the proofs of the 
iteration lemmas for preservation of semiproperness through limit stages.
This might be partially explained by the fact that
 the proof of the properness of countable support iteration of proper posets is well understood and 
 the modifications required to handle the proof of semiproperness for 
 countable support iterations of length at most $\omega_1$ is obtained from that proof
 with minor variations.
 Nonetheless we believe that our ``algebraic'' treatment of iterated forcing gives a simpler 
 and more elegant presentation of the whole theory of iterated forcing 
 and outlines neatly the connections between the notion of
 properness and semiproperness in the theory of forcing, and their Baire category counterparts in topology.
 Moreover we believe that our presentation opens the way to 
 handle different kinds of limits given by complete iteration systems indexed by arbitrary partial orders
 and also to develop a theory of iterated forcing for stationary set preserving posets.
 We are also curious to see if this approach could simplify the treatment of semiproper iterations which
 do not add reals and may help to foresee a fruitful theory of iterated forcing
 which preserve $\aleph_1$ and $\aleph_2$.

\subparagraph*{Acknowledgements}

These notes developed out of a course officially held by the first author, 
nonetheless the contributions of the audience to a successful outcome of the course is unvaluable.
In particular all the proofs presented in these notes are the fruit of a joint elaboration of (at least)
all three authors of these notes plus (in many cases) also of the other participants to the course.
For these reasons we thank in particular:
\begin{itemize}
\item
Rapha\"{e}l Carroy for for his contributions to sections~\ref{rap1} 
and~\ref{rap2} on Stone spaces and the topological characterization
of properness and semiproperness. The observations there contained
are the outcome of discussions between him and the first author.

\item
Fiorella Guichardaz whose master thesis contains the bulk of results on 
which sections  \ref{sec:regemb} to \ref{sec:twostepitquot} expand.
\item
Daisuke Ikegami for the several advices he has given during the redaction of these notes.
\item
Bruno Li Marzi for his keen interest on the subject.
\end{itemize}

%% file: 1_Posets.tex
\section{Posets, boolean algebras, forcing} \label{sec:background}
In this section we present some general facts that are required for the development of the remainder of these notes. 
Reference texts for this section are~\cite{BELL},~\cite{JECH}, and~\cite{KUNEN}.

\subsection{Posets and boolean algebras}
We introduce posets and complete boolean algebras and we prove that forcing equivalent notions have isomorphic boolean completions.
This allows us to focus on complete boolean algebras. 

	\begin{definition}
		A \emph{poset} (partially ordered set) is a set $P$ together with a binary relation $\leq$ on $P$ which is transitive, reflexive and antisymmetric.
	\begin{itemize}
	\item
		Given $a, b \in P$, $a \perp b$ ($a$ and $b$ are \emph{incompatible})
		if and only if:
		\[
		a \perp b \quad \Leftrightarrow \quad \neg \exists c: ~ c \leq a \wedge c \leq b
		\]
		Similarly, $a \parallel b$ ($a, b$ are \emph{compatible})
		iff $\neg \cp{a \perp b}$.
	\item
		A subset $A \subset P$ is a \emph{chain} if and only if is totally ordered in $P$.
	\item	
		A subset $A \subset P$ is an \emph{antichain} if and only if the elements of $A$ are pairwise incompatible.
	\item	
		An antichain $A \subset P$ is \emph{maximal} if and only if no strict superset 
		of $A$ is an antichain.
	\item For a given $A\subset P$, 
			\[
				\downarrow{A}=\{p:\exists q\in A, p\leq_P q\},
			\]
		  and
			\[
		    	\uparrow{A}=\{p:\exists q\in A, p\geq_P q\}.
			\]
	\item
    	A set $D \subset P$ is dense iff $\uparrow D =P$. 
	\item
	    A set $B \subset P$ is predense iff  $\downarrow B$ is dense.
	\item A set $B \subset P$ is directed iff 
	\[
		\forall p,\ q \in B \exists r \in B (p \leq r \wedge q \leq r).
	\]
	\item
	A poset $P$ is separative iff for all $p\not\leq q \in P$, there exists $r \in P$ with $r \leq p$, $r \perp q$.
	\item
	A poset $P$ is ${<}\lambda$-cc (chain condition) iff $|A|<\lambda$ for all maximal antichains $A\subset P$.	
	\end{itemize}
	\end{definition}
	
	
	\begin{fact}
		Every dense set $D \subset P$ contains a maximal antichain $A \subset D$. Conversely, any maximal antichain is predense.
	\end{fact}

	\begin{definition}
		A set $I \subset P$ is an \emph{ideal} in $P$ iff it is downward closed and upward directed
		(i.e. it has upper bounds for all of its finite subsets).
		A set $F \subset P$ is a \emph{filter} in $P$ iff it is upward closed and downward directed.	\end{definition}

	\begin{definition}
		Let $M$ be a model of $\ZFC$ and $P \in M$ be a poset. A filter $G \subset P$ is $M$-generic for $P$ if and only if $G\cap D\cap M\neq\emptyset$ for every dense set $D$ of $P$ in $M$. 
	Equivalently, a filter $G$ is $M$-generic if it intersects inside $M$ 
	every maximal antichain of $P$ in $M$.
	\end{definition}

	\begin{fact}
		If $P \in M$ is a separative poset, no filter in $M$ is $M$-generic. 
	\end{fact}

	\begin{definition}
		A poset $P$ is a \emph{lattice} if any two elements $a,b$ have a unique supremum $a \vee b$ (least upper bound, join) and infimum $a \wedge b$ (greatest lower bound, meet).
	\begin{itemize}
	\item
		A lattice is \emph{distributive} if the operations of join and meet distribute over each other.
	\item
		A lattice $\mathbb{L}$ is \emph{bounded} if it has a least element ($\0_{\mathbb{L}}$) 
		and a greatest element ($\1_{\mathbb{L}}$).
	\item
		A lattice is \emph{complemented} if it is bounded lattice and every element $a$ has a complement, i.e. an element $\neg a$ satisfying $a \vee \neg a = \1$ and $a \wedge \neg a = \0$.
		\item
		A \emph{boolean algebra} is a complemented distributive lattice. A boolean algebra is \emph{complete} iff every subset has a supremum and an infimum.
		\end{itemize}
	\end{definition}
	
	\begin{fact}
	Given a boolean algebra $\BB$,
	$\BB \setminus \bp{\0}=\BB^+$ is a separative poset 
	with the order relation $a\leq_{\BB^+}b$ given by any of the following requirement on $a$ and $b$:
	\begin{itemize}
	\item
	$a\wedge b=a$,
	\item
	$a\vee b=b$. 
	\end{itemize}
\end{fact}

	To any poset we can associate a unique (up to isomorphism) boolean completion:

	\begin{theorem}
		For every poset $P$ there exists a unique (up to isomorphism) complete boolean algebra 
		$\BB$ (the boolean completion of $P$) with a \emph{dense embedding} 
		$i_P: P \to \BB^+$ such that for any $p,q \in P$:
		\begin{itemize}
			\item $p \leq q \implies i_P(p) \leq i_P(q)$,
			\item $p \perp q \implies i_P(p) \perp i_P(q)$,
			\item $i_P[P]$ is a dense subset of $\BB^+$.
		\end{itemize}
	\end{theorem}

\begin{proof}
We briefly sketch how to define the boolean completion of a poset $P$.

For $A,B\subset P$, say that $A\leq^*_P B$ if $(\downarrow{A})\cap (\downarrow{B})$ is a dense subset of $(\downarrow{A},\leq_P\cap (\downarrow{A})^2 )$.

Remark that $\leq^*_P$ is a transitive reflexive relation on $\pp(P)$ and thus we can define
an equivalence relation $\equiv^*_P$ by $A\equiv^*_P B$ if $A\leq^*_P B\leq^*_P A$.

Then observe that $\leq^*_P$ induces a separative partial order on $\pp(P)/\equiv^*_P$ and that
for any $A\subset P$, 
$\bigcup[A]_{\equiv^*_P}$ is the maximal element under inclusion of
$[A]_{\equiv^*_P}$.

Now define
\begin{itemize}
\item
$\bigwedge \{[A_i]_{\equiv^*_P}:i\in I\}=[\bigcap \{\bigcup [A_i]_{\equiv^*_P}:i\in I\}]_{\equiv^*_P}$,
\item
$\bigvee \{[A_i]_{\equiv^*_P}:i\in I\}=[\bigcup \{\bigcup [A_i]_{\equiv^*_P}:i\in I\}]_{\equiv^*_P}$,
\item
$\neg [A]_{\equiv^*_P}=[\{p\in P: (\downarrow{\{p\}})\cap (\bigcup[A]_{\equiv^*_P})=\emptyset\}]_{\equiv^*_P}$.
\end{itemize}

We leave to the reader to check that the above operations make $\pp(P)/\equiv^*_P$ a complete boolean algebra and that
the map $i_P:P\to \pp(P)/\equiv^*_P$ which maps $p$ to  $[\{p\}]_{\equiv^*_P}$ is a complete
embedding with a dense image.
\end{proof}

	\begin{definition}
	Let $V$ be a transitive model of $\ZFC$ and $\BB$ be a complete boolean algebra in $V$.
		A set $U \subset \BB$ is an \emph{ultrafilter} in $\BB$ if and only if $U$ is a filter and for any 
		$b \in \BB$, $b \in U$ or $\neg b \in U$.
	
		If $I$ is an ideal of $\BB$, the \emph{quotient} $\BB / I$ is the quotient of $\BB$ with 
		respect to the equivalence relation defined by $a \approx b \Leftrightarrow a \triangle b \in I$
		($a\triangle b=(a \vee b)\wedge \neg(a\wedge b)$).
	\end{definition}

	\begin{notation}
	Let $\BB$ be a complete boolean algebra. $X\subset\BB$ is a pre-filter if $\uparrow X$ is a filter and is a pre-ideal if $\downarrow X$ is an ideal. Given $X\subset\BB$, we let $I^*=\{\neg a:a\in I\}$. It is well known that $I^{**}=I$ and $I$ is an ideal iff $I^*$ is a filter. 
	With an abuse of notation, if $G$ is a pre-filter on $\BB$, we write $\BB/G$ also to denote
	$\BB/(\uparrow G)^*$.
	
	If $I\subset\BB$ is an ideal $\BB/I$ is always a boolean algebra (but in general it is not complete).
	\end{notation}

\subsection{Stone spaces and dual properties.} \label{rap1}

Complete boolean algebras can be related to topological spaces with certain specific properties.

\begin{definition}
The closure $\overline{U}$ of a subset $U$ of a topological space $(X, \tau)$ is the intersection of all closed sets containing $U$.

The interior $\mathring{U}$ of a subset $U$ of a topological space $(X, \tau)$ is the union of all open sets contained in $U$.
\end{definition}

\begin{definition}\label{def:ROP}
		Let $(X,\tau)$ be a topological space. 
		The \emph{regular open algebra} of $X$ is 
		$\mathrm{RO}(X,\tau) = \bp{ U \subset X: ~ U = \mathring{\overline{U}} }$ 
		ordered by set-theoretical inclusion. $\RO(X,\tau)$ is a complete boolean algebra 
		with the operations defined as follows:
	\begin{itemize}
\item
$U\wedge V=U\cap V$ for all $U,V\in\mathrm{RO}(X,\tau)$,
\item
$\neg U= \mathring{\overline{X\setminus U}}$ for all  $U\in\mathrm{RO}(X,\tau)$,
\item
$\bigwedge A=\mathring{\overline{\bigcap A}}$ for all  $A\subset\mathrm{RO}(X,\tau)$,
\item
$\bigvee A=\mathring{\overline{\bigcup A}}$ for all  $A\subset\mathrm{RO}(X,\tau)$.
\end{itemize}

The elements of $\RO(X,\tau)$ are the regular open subsets of $X$.

Given $\BB$ a complete boolean algebra, its \emph{Stone space} is
\[
X_{\BB}=\{G \subseteq \BB : G \mbox{ is an ultrafilter }\}.
\]
$X_{\BB}$ is endowed with the topology $\tau_{\BB}$ generated by
\[
\{N_b=\{G \in X_\BB : b \in G\}: b \in \BB\}.
\]

Given a poset $Q$, let $(X_Q,\tau_Q)$ be the Stone space associated to its boolean completion
 $\pp(Q)/\equiv^*_Q$.

\end{definition}

We can now spell out the relation between a poset, its boolean completion and the
complete boolean algebra given by the regular open set of the corresponding Stone space.

\begin{remark}

Let $\BB$ be a complete boolean algebra and $(X_{\BB},\tau_{\BB})$ be its associated Stone space.
It is well-known that:
\begin{itemize}
	\item each basic open set is also a closed set, since it is the complement of $N_{\neg b}$;
	\item every clopen set is a regular open, thus each $N_b$ is a regular open;
	\item $\BB$ is a complete boolean algebra, thus $\{N_b: b \in \BB\}$ is the family of the regular open sets of $\tau_{\BB}$;
	\item if $A \subseteq \BB$ then $N_{\bigvee A} = \mathring{\overline{\bigcup \{N_p:p\in A\}}}$;
	\item $X_\BB$ is Haussdorff and compact.
\end{itemize}

Let $Q$ be a poset and $(X_Q,\tau_Q)$ be its associated topological space defined above.
We can check that for any $A\subset Q$, 
\[
X_A=\bigcup\{N_p:p\in A\}
\]
is a regular open set in $\tau_Q$ iff
$A=\bigcup[A]_{\equiv^*_Q}$ and 
that $A\equiv^*_Q B$ iff $\mathring{\overline{X_A}}=\mathring{\overline{X_B}}$.

The above observation allows to define a natural isomorphism 
\[
\Phi_Q: ( \pp(Q)/\equiv^*_Q)  \to  \RO(X_Q,\tau_Q)
\]
defined by $[A] \mapsto  X_{\bigcup[A]}$.

Moreover
the separative quotient of $Q$ is mapped by the isomorphism in the basis $\{N_q:q\in Q\}$ for $\tau_Q$
which is a dense subset of $\RO(X_Q,\tau_Q)^+$.
\end{remark}

In view of the above remark we are led to the following:
\begin{notation}
For any given poset $Q$ we let $\RO(Q):=\RO(X_Q,\tau_Q)$ denote its boolean completion.
\end{notation}

%
%
%
%
%

\subsection{Forcing and boolean valued models}

We assume the reader is a acquainted with the basic development of forcing and boolean valued models, here we resume the result and definitions we shall need in the form which is more convenient for us.

\begin{definition}\label{def:forcingnames}
Let $V$ be a transitive model of $\ZFC$ and $\BB$ be a complete boolean algebra in $V$. 
\begin{equation*}
	V^{\BB}=\{\dot{a}\in V: \dot{a}:V^{\BB}\to \BB\text{ is a partial function}\footnotemark\}.
\end{equation*}
\footnotetext{This definition is a shorthand for a recursive definition by rank. We remark that in certain cases (for example in the definition of the name $\dot{\beta}$ in the proof of 
Lemma~\ref{LemmaGenerale} below) it will be convenient to allow a name $\dot{a}$ to be a relation 
(as in Kunen's~\cite[Definition 2.5]{KUNEN}); given a $\BB$-name $\dot{a}$ according to Kunen's definition, the corresponding intended name $f_{\dot{a}}$ according to the above definition 
is given by $f_{\dot{a}}(f_{\dot{c}}) = \bigvee \bp{b: ~ \ap{\dot{c},b} \in \dot{a}}$.}
We let for the atomic formulas $x\in y$, $x\subseteq y$, $x=y$:
\begin{itemize}
\item
$\Qp{\dot{b}_0\in\dot{b}_1}_{\BB}=\bigvee\bp{\Qp{\dot{a}=
\dot{b}_0}_{\BB}\wedge\dot{b}_0(\dot{a}): {\dot{a}\in\dom(\dot{b}_1)}}$,
\item
$\Qp{\dot{b}_0\subseteq\dot{b}_1}_{\BB}=\bigwedge\bp{\neg\dot{b}_0(\dot{a})\vee
\Qp{\dot{a}\in
\dot{b}_0}_{\BB}:{\dot{a}\in\dom(\dot{b}_0)}}$,
\item
$\Qp{\dot{b}_0=\dot{b}_1}_{\BB}=\Qp{\dot{b}_0\subseteq\dot{b}_1}_{\BB}\wedge\Qp{\dot{b}_1\subseteq\dot{b}_0}_{\BB}$.
\end{itemize} 
For general formulas $\phi(x_0,\dots,x_n)$, we let:
\begin{itemize}
\item
$\Qp{\neg\phi}_{\BB}=\neg\Qp{\phi}_{\BB}$,
\item
$\Qp{\phi\wedge\psi}_{\BB}=\Qp{\phi}_{\BB}\wedge\Qp{\psi}_{\BB}$,
\item
$\Qp{\phi\vee\psi}_{\BB}=\Qp{\phi}_{\BB}\vee\Qp{\psi}_{\BB}$,
\item
$\Qp{\exists x\phi(x,\dot{b}_1,\dots,\dot{b}_n)}_{\BB}=
\bigvee\bp{\Qp{\phi(\dot{a},\dot{b}_1,\dots,\dot{b}_n)}_{\BB}: {\dot{a}\in V^{\BB}}}$.
\end{itemize}

\end{definition}

When the context is clear, we will omit the index.

\begin{notation}
For a complete boolean algebra $\BB$,
$\dot{G}_{\BB}\in V^{\BB}$ always denote the canonical name for a $V$-generic filter for $\BB$, i.e.
\[
\dot{G}_{\BB}=\{\langle \check b,b\rangle: b\in\BB\}.
\]
\end{notation}

\begin{theorem}[\L o\'s]
Let $V$ be a transitive model of $\ZFC$ and $\BB$ be a complete boolean algebra in $V$.
Let $G$ be any ultrafilter on 
$\BB$. For $\dot{a},\dot{b}\in V^{\BB}$ we let 
\begin{itemize}
\item
$\dot{b} =_G \dot{a}$ iff $\Qp{\dot{b} =\dot{a}}\in G$,
\item
$[\dot{b}]_G=\{\dot{a}:\dot{b} =_G\dot{a}\}$, 
\item
$[\dot{b}]_G \in_G [\dot{a}]_G$ iff $\Qp{\dot{b} \in\dot{a}}\in G$,
\item
$V^{\BB}/G=\{[\dot{b}]_G:\dot{b}\in V^{\BB}\}$.
\end{itemize}
Then:
\begin{enumerate}
\item
$(V^\BB/G,\in_G)$ is a model of $\ZFC$ 
\item
$(V^\BB/G,\in_G)$ models $\phi([\dot{b}_1]_G,\dots,[\dot{b}_n]_G)$ iff
$\Qp{\phi(\dot{b}_1,\dots,\dot{b}_n)}\in G$.
\end{enumerate}
\end{theorem}

\begin{definition}
Let $V$ be transitive model of $\ZFC$, $\BB\in V$ be a complete boolean algebra in $V$,
$G$ be a $V$-generic filter for $\BB$.
For any $\dot{b}\in V^{\BB}$ we let 
\[
\dot{b}_G=\{\dot{a}_G:\exists p\in G\langle\dot{a},p\rangle\in\dot{b}\}.
\]
$V[G]=\{\dot{b}_G:\dot{b}\in V^P\}$.
\end{definition}

\begin{theorem}[Cohen's forcing theorem] 
Let $V$ be transitive model of $\ZFC$, $\BB\in V$ be a complete boolean algebra
$G$ be a $V$-generic filter for $\BB$. Then:
\begin{enumerate}
\item
$V[G]$ is isomorphic to $V^{\BB}/G$ via the map
which sends $\dot{b}_G$ to $[\dot{b}]_{G}$.
\item
$V[G]\models\phi((\dot{b}_1)_G,\dots,(\dot{b}_n)_G)$ iff
$\Qp{\phi(\dot{b}_1,\dots,\dot{b}_n)}\in G$.
\item
$b\leq_{\BB} \Qp{\phi(\dot{b}_1,\dots,\dot{b}_n)}$ iff
$V[G]\models\phi((\dot{b}_1)_G,\dots,(\dot{b}_n)_G)$
 for all $V$-generic filters
$G$ for $\BB$ such that $b\in G$.
\end{enumerate}
\end{theorem}

\begin{notation}
Given a partial order $P$ and $\dot{b}_1,\dots,\dot{b}_n\in V^{\RO(P)}$ we say that
$p\Vdash_P\phi(\dot{b}_1,\dots,\dot{b}_n)$ iff 
$i_P(p)\leq\Qp{\phi(\dot{b}_1,\dots,\dot{b}_n)}$.
\end{notation}

\begin{lemma}[Mixing]
Let $\BB$ be a complete boolean algebra and
$\{\dot{b}_a:a\in A\}$ be a family of $\BB$-names indexed by an antichain.
Then there exists $\dot{b}\in V^{\BB}$ such that
$\Qp{\dot{b}=\dot{b}_a}\geq a$ for all $a\in A$.
\end{lemma}

\begin{lemma}[Fullness]
Let $\BB$ be a complete boolean algebra. 
For all formula $\phi(x,x_1,\dots,x_n)$ and 
$\dot{b}_1,\dots,\dot{b}_n\in V^{\BB}$, there is 
$\dot{b}\in V^{\BB}$ such that
\[
\Qp{\exists x\phi(x,\dot{b}_1,\dots\dot{b}_n)}=\Qp{\phi(\dot{b},\dot{b}_1,\dots\dot{b}_n)}.
\]
\end{lemma}

\begin{fact}
Let $V$ be transitive model of $\ZFC$, $\BB\in V$ be a complete boolean algebra,
$G$ be a $V$-generic ultrafilter for $\BB$. Then $\bigwedge A\in G$
for any $A\subset G$ which belongs
to $V$.
\end{fact}

%% file: 2_Embeddings.tex
\section{Regular embeddings} \label{sec:regemb}

\subsection{Embeddings and retractions}

In this part we introduce the notions of complete homomorphism and regular embedding and their basic properties.

\begin{definition}
	Let $\BB$, $\QQ$ be complete boolean algebras, $i: \BB \to \QQ$ is a \emph{complete homomorphism} iff it is an homomorphism that preserves arbitrary suprema. We say that $i$ is a \emph{regular embedding} iff it is an injective complete homomorphism of boolean algebras.
\end{definition}

\begin{definition}
	Let $\BB$ be a complete boolean algebra, and let $b \in \BB$. Then 
	\[
		\BB\res b = \{c \in \BB : c \leq b \},
	\]
	and 
	\[ 	
		\begin{split}
		\BB &\to \BB \res b \\
		c &\mapsto  c \wedge b.
		\end{split}
	\]
	is 	the restriction map from $\BB$  to $\BB\res b$.
\end{definition}

\begin{definition} \label{eCoker}
	Let $i: \BB \to \QQ$ be a complete homomorphism. We define
	\[
	\ker(i) = \bigvee \bp{b \in \BB: \ i(b) = \0_\QQ}
	\]
	\[
	\coker(i) = \neg\ker(i)	
	\]
\end{definition}

\begin{remark}
We can always factor a complete homomorphism $i: \BB \to \QQ$ as the restriction map from $\BB$ to 
$\BB \res \coker(i)$ (which we can trivially check to be a complete and surjective homomorphism)
composed with the regular embedding $i \res \coker(i)$.
This factorization allows to generalize easily many results on regular embeddings to results on
complete homomorphisms.
\end{remark}

\begin{definition}
	Let $i: \BB \to \QQ$ be a regular embedding, the \emph{retraction} associated to $i$ is the map
	\[
	\begin{array}{llll}
		\pi_i :& \QQ &\to& \BB \\
		&c &\mapsto& \bigwedge \bp{b \in \BB: ~ i(b) \geq c}
	\end{array}
	\]
\end{definition}

\begin{proposition} \label{eRetrProp}
	Let $i:\BB\to\QQ$ be a regular embedding, $b \in \BB$, $c,d \in \QQ$ be arbitrary. Then,
	\begin{enumerate}
		\item $\pi_i\circ i(b)= b$ hence $\pi_i$ is surjective;
		\item \label{eRPComp} $i\circ\pi_i(c)\geq c$ hence $\pi_i$ maps $\QQ ^+$ to $\BB^+$;
		\item \label{eRPJoins} $\pi_i$ preserves joins, i.e. $\pi_i (\bigvee X)= \bigvee \pi_i[X]$ for all $X \subseteq \QQ$;
		\item $i(b)= \bigvee \{e : \pi_i (e) \leq b \}$.
		\item \label{eRPHomo} $\pi_i (c \wedge i(b)) = \pi_i(c)\wedge b = \bigvee \{ \pi_i(e): e \leq c, \pi_i(e) \leq b\}$;
		\item \label{eRPMeets} $\pi_i$ does not preserve neither meets nor complements whenever $i$ is not surjective, but $\pi_i(d \wedge c) \leq \pi_i(d) \wedge \pi_i(c)$ and $\pi_i(\neg c) \geq \neg \pi_i(c)$;
	\end{enumerate}
\end{proposition}

\begin{proof}

\begin{enumerate}
	\item We have
	\[
 		 \pi_i\circ i(b) =\bigwedge \{a\in {\BB}:i(a)\geq i(b)\})=
		 \bigwedge \{a\in {\BB}:a\geq b\})=b,
	\]
	since $i$ is injective. Hence $\pi$ is surjective.
	\item We have
	\[
	\begin{split}
		 i\circ\pi_i(c)  &= i\cp{\bigwedge \{b\in {\BB}:i(b)\geq c\}}=\bigwedge \{i(b):b\in {\BB},\,i(b)\geq c\} \\
         &\geq \bigwedge \{d\in {\QQ}:d\geq c\} =c.
    \end{split}
	\]
    Now if, by contradiction, for $c > \0_{\QQ},$ $\pi_i(c)=\0_{\BB},$  we would have $\0_{\QQ}=i\circ \pi (c) \geq c > \0_{\QQ}.$
	\item Let $X= \{ c_j:j \in J\} \subseteq \QQ.$ Thus, for all $k\in J,$
	\[
		\begin{split}
 			 \pi_i\cp{\bigvee \{c_j:j\in J\}} &= \bigwedge \{b\in {\BB}:i(b) \geq \bigvee \{c_j:j\in J\} \} \\
  			&\geq \bigwedge \{b\in {\BB}: i(b)\geq c_k\} =\pi_i(c_k).
		\end{split}
	\]
	In that way, we obtain the first inequality: $\pi_i(\bigvee X) \geq \bigvee \pi_i[X].$ 

	Now if $a = \bigvee \pi_i[X],$ we have that $a\geq \pi_i(c_j)$ for all $j\in J.$ Thus, for all $j\in J:$
	\[
		i(a)\geq i\circ \pi_i(c_j)\geq c_j .
	\]
	In particular $i(a)\geq \bigvee \{c_j:j\in J\}.$ By definition, $\pi_i$ is increasing, so:
	\[
		a = \pi_i(i(a)) \geq \pi\cp{\bigvee \{c_j:j\in J\}},
	\]
	that is, the second inequality $\bigvee \pi_i[X] \geq \pi_i (\bigvee X)$ holds.
	\item 
	Let $e \in \QQ$ such that $\pi_i(e) \leq b.$ Since $i$ is order preserving, 
	$s \leq i(\pi_i(e)) \leq i(b).$
	 Thus
	\[
		 \bigvee \{e: \pi_i(e) \leq b \} \leq i(b).
	\]
	In order to prove the other inequality; recall that $b= \pi_i(i(b)).$ So 
	\[
 	 i(b)  \leq \bigvee \{e : \pi_i (e) \leq \pi_i (i(b)) \}= \bigvee \{e : \pi_i(e) \leq b \}. 
	\]
	\item For $b \in \BB, c \in \QQ,$ the following three equations hold:
	\begin{align} 
		\label{piq1} \pi_i(c \wedge i(b)) \vee \pi_i (c \wedge \non i(b)) & = \pi_i(c); \\
		\label{piq2} (\pi_i(c) \wedge b) \vee (\pi_i(c) \wedge \non b) & = \pi_i(c);\\
		\label{piq3} (\pi_i(c) \wedge b) \wedge (\pi_i(c)\wedge \non b) & = \0_{\BB}.
	\end{align}
	Furthermore, by $\pi_i$ definition, we have:
	\begin{equation} \label{min1}
		 \pi_i(c \wedge i(b)) \leq \pi_i(c) \wedge b;
	\end{equation}
	\begin{equation} \label{min2}
		 \pi_i(c \wedge \non i (b)) = \pi_i (c \wedge i(\non b) ) \leq \pi_i(c) \wedge \non b. 
	\end{equation}
	By (\ref{min1}), (\ref{min2}), and (\ref{piq3}) we get
	\[ 
		 \pi_i(c \wedge i(b)) \wedge \pi_i(c \wedge \non i (b)) = (\pi_i(c) \wedge b) \wedge (\pi_i(c) \wedge \non b)=  \0_{\BB}.
	\]
	Moreover, by (\ref{piq1}) and (\ref{piq2}),
	\[
		 \pi_i(c \wedge i(b)) \vee \pi_i(c \wedge \non i (b)) = (\pi_i(c) \wedge b) \vee (\pi_i(c) \wedge \non b).
	\]

	All in all, we conclude that
	\[
		 \pi_i(c \wedge i(b))=\pi_i(c) \wedge b \mbox{ and } \pi_i(c \wedge \non i (b))=\pi_i(c) \wedge \non b.
	\]
	\item If $i:\BB \rightarrow \QQ$ is not surjective, then pick $c \in \QQ \setminus i [\BB].$
	Then $i (\pi_i (c)) \neq c$ and we have $i (\pi_i (c)) > c.$ Thus $d = i(\pi_i(c)) \wedge \neg c > \0_{\QQ}$ and $ \pi_i(d) > \0_{\BB}.$ Now,
	\[
		\pi_i(c) \vee \pi_i(d)  = \pi_i (c \vee d) = \pi_i (i (\pi_i(c)))= \pi_i (c).
	\]
	Thus $\pi_i (d) \wedge \pi_i (c) = \pi_i (d) > \0_{\BB}.$ But $\pi_i(d \wedge c) = \pi_i (\0_{\QQ})= \0_{\BB},$ so $\pi_i$ does not preserve meets. It cannot preserve complements, since it preserves joins and so otherwise it should preserve meets. 
	
	However for all $d$, $c \in \QQ$:
	\[\pi_i (d \wedge c) \leq \pi_i(d \wedge i (\pi_i(c)))= \pi_i(d) \wedge \pi_i(c);\]
	and $\neg\pi_i(d) \leq \pi_i(\neg d)$, since
	\[\neg \pi_i(d) \wedge \neg \pi_i(\neg d)= \neg (\pi_i(d) \vee \pi_i(\neg d)) = \neg (\pi_i(d \vee \neg d)) = \neg (\pi_i(\1))= \0.\]
\end{enumerate}
\end{proof}

Complete homomorphism and regular embeddings are the boolean algebraic counterpart of
two-step iterations, this will be spelled out in detail in section~\ref{sec:twostepitquot}.
Below we outline the relation existing between generic extensions by $\BB$ and $\QQ$ in case
there is a complete homomorphism $i:\BB\to\QQ$.

\begin{lemma} \label{eRetrDense}
	Let $i:\BB\to\QQ$ be a regular embedding, $D \subset \BB$, $E \subset \QQ$ be predense sets, then $i[D]$ and $\pi_i[E]$ are predense (i.e. predense subsets are mapped into predense subsets). Moreover $\pi_i$ maps $V$-generic filter in $V$-generic filters.
\end{lemma}

\begin{proof}
	First, let $c \in \QQ$ be arbitrary. Since $D$ is predense, there exists $d \in D$ such that $d \wedge \pi(c) > \0$. Then by Property \ref{eRetrProp}.\ref{eRPHomo} also $i(d) \wedge c > \0$ hence $i[D]$ is predense.
	Finally, let $b \in \BB$ be arbitrary. Since $E$ is predense, there exists $e \in E$ such that $e \wedge i(b) > \0$. Then by Property \ref{eRetrProp}.\ref{eRPHomo} also $\pi_i(e) \wedge b > \0$ hence $\pi_i[E]$ is predense.

	For the last point in the lemma, we first prove that $\pi_i[G]$ is a filter whenever $G$ is a filter. Let $c$ be in $G$, and suppose $b > \pi_i(c) $. Then by Property \ref{eRetrProp}.\ref{eRPComp} also 
	$i(b) > i(\pi_i(c)) \geq c$, hence $i(b) \in G$ and $b \in \pi_i[G]$, proving that $\pi_i[G]$ is upward closed. Now suppose $a, c \in G$, then by Property \ref{eRetrProp}.\ref{eRPMeets} we have that $\pi_i(a) \wedge \pi_i(c) \geq \pi_i(a \wedge c) \in \pi_i[G]$ since $a \wedge c \in G$. Combined with the fact that $\pi_i[G]$ is upward closed this concludes the proof that $\pi_i[G]$ is a filter.

	Finally, let $D$ be a predense subset of $\BB$ and assume $G$ is $V$-generic for $\QQ$. 
	We have that $i[D]$ is predense hence $i[D] \cap G \neq \emptyset$ by $V$-genericity of $G$. Fix $c \in i[D] \cap G$, then $\pi_i(c) \in D \cap \pi_i[G]$ concluding the proof.
\end{proof}

\begin{lemma}
	Let $i:\BB\to\QQ$ be an homomorphism of boolean algebras. Then $i$ is a complete homomorphism iff for every $V$-generic filter $G$ for $\QQ$, $i^{-1}[G]$ is a $V$-generic filter for $\BB$.
\end{lemma}

\begin{proof}
	If $i$ is a complete homomorphism and $G$ is a $V$-generic filter, then $i^{-1}[G]$ is trivially a filter. Furthermore, given $D$ dense subset of $\BB$, $i[D]$ is predense so there exists a $c \in G \cap i[D]$, hence $i^{-1}(c) \in i^{-1}[G] \cap D$.
	
	Conversely, suppose by contradiction that there exists an $A \subseteq \BB$ such that $i(\bigvee A) \neq \bigvee i[A]$ (in particular, necessarily $i(\bigvee A) > \bigvee i[A]$). Let $d = i(\bigvee A) \setminus \bigvee i[A]$, $G$ be a $V$-generic filter with $d \in G$. Then $i^{-1}[G] \cap A = \emptyset$ hence is not $V$-generic below $\bigvee A \in i^{-1}[G]$, a contradiction.
\end{proof}

Later in these notes we will use the following lemma to produce local versions of various results.

\begin{lemma}[Restriction]\label{6StepRestr}
	Let $i:\BB\to\QQ$ be a regular embedding, $c \in \QQ$, then
	\[
	\begin{array}{cccc}
	i_c: & \BB \res \pi_i(c) &\to& \QQ \res c \\
	& b &\mapsto& i(b) \wedge c
	\end{array}
	\]
	is a regular embedding and its associated retraction is $\pi_{i_c} = \pi_i\res (\QQ\res c)$.
\end{lemma}

\begin{proof}
	First suppose that $i_c(b) = \0$, then by Proposition \ref{eRetrProp}.\ref{eRPHomo},
	\[
	\0 = \pi_i(i_c(b)) = \pi_i(i(b) \wedge c ) = b \wedge \pi_i(c) = b
	\]
	that ensures the regularity of $i_c$. Furthermore, for any $d \leq c$,
	\[
	\begin{array}{lll}	
	\pi_{i_c}(d) &=& \bigwedge \bp{b \leq \pi_i(c): ~ i(b) \wedge c \geq d } \\
	&=& \bigwedge \bp{b \leq \pi_i(c): ~ i(b) \geq d } = \pi_i(d),
	\end{array}
	\]
	concluding the proof.
\end{proof}

The notion of regular embedding and associated retraction can also be translated in the context of Stone spaces. Recall that for a complete boolean algebra $\BB$, $X_{\BB}$ is the Stone space of $\BB$ whose points are the ultrafilters on $\BB$ and whose topology is generated by the class of regular open sets
\[
N_a=\{G\in X_{\BB}:a\in \BB\}.
\]
\begin{proposition}
	Let $i: \BB \rightarrow \QQ$ a regular embedding of complete boolean algebras. 
	
	Then the following map:
	\[
	\begin{aligned}
		\pi^*: X_\QQ &\rightarrow X_\BB\\
		G &\mapsto \pi_i[G],
	\end{aligned}
	\]
	is continuous and open (since $\pi^*[N_c]=N_{\pi_i(c)}$). Moreover, 
	$X_\QQ /_\approx \simeq X_\BB$, where
	\[
	G \approx H \iff \pi^*(G)=\pi^*(H).
	\]
\end{proposition}

\subsection{Embeddings and boolean valued models}

Complete homomorphisms of complete boolean algebras extend to natural $\Delta_1$-elementary 
maps between boolean valued models.

\begin{proposition}
	Let $i:\BB\to\QQ$ be a complete homomorphism, and define by recursion $\hat{\imath}: V^\BB \to V^\QQ$ by
	\[
	\hat{\imath}(\dot{b})(\hat{\imath}(\dot{a}))=i\circ\dot{b}(\dot{a})
	\]
	for all $\dot{a}\in\dom(\dot{b})\in V^{\BB}$.
	Then the map $\hat{\imath}$ is $\Delta_1$-elementary, i.e. for every $\Delta_1$ formula $\phi$,
	\[
	i\cp{ \Qp{ \phi(\dot{b}_1,\ldots,\dot{b}_n) }_\BB } = \Qp{ \phi( \hat{\imath}(\dot{b}_1),\ldots,\hat{\imath}(\dot{b}_n) ) }_\QQ
	\]
\end{proposition}

\begin{proof}
	We prove the result by induction on the complexity of $\phi$. For atomic formulas $\psi$ (either $x = y$ or $x \in y$), we proceed by further induction on the rank of $\dot{b}_1$, $\dot{b}_2$.
	
	\[
	\begin{split}
		i\cp{ \Qp{\dot{b}_1 \in \dot{b}_2}_\BB } &= i\cp{ \bigvee \bp{\dot{b}_2(\dot{a}) \wedge \Qp{\dot{b}_1 = \dot{a}}_\BB : \dot{a} \in \dom(\dot{b}_2)}} \\
		&= \bigvee \bp{ i\cp{\dot{b}_2(\dot{a})} \wedge i\cp{ \Qp{\dot{b}_1 = \dot{a}}_\BB }: {\dot{a} \in \dom(\dot{b}_2)}} \\
		&= \bigvee\bp{ i\cp{\dot{b}_2(\dot{a})} \wedge \Qp{\hat{\imath}(\dot{b}_1) = \hat{\imath}(\dot{a})}_\QQ : {\dot{a} \in \dom(\dot{b}_2)}}\\
		&= \Qp{\hat{\imath}(\dot{b}_1) \in \hat{\imath}(\dot{b}_2)}_\QQ \\
		i\cp{ \Qp{\dot{b}_1 \subseteq \dot{b}_2}_\BB } &= i\cp{ \bigwedge \bp{ \dot{b}_1(\dot{a}) \rightarrow \Qp{\dot{a} \in \dot{b}_2}_\BB: {\dot{a} \in \dom(\dot{b}_1)}}} \\	
		&= \bigwedge \bp{ i\cp{\dot{b}_1(\dot{a})}\rightarrow i\cp{ \Qp{\dot{a} \in \dot{b}_2}_\BB }: {\dot{a} \in \dom(\dot{b}_1)}} \\
		&= \bigwedge \bp{ i\cp{\dot{b}_1(\dot{a})} \rightarrow \Qp{\hat{\imath}(\dot{a}) \in \hat{\imath}(\dot{b}_2)}_\QQ: {\dot{a} \in \dom(\dot{b}_1)}} \\
		&=\Qp{\hat{\imath}(\dot{b}_1) \subseteq \hat{\imath}(\dot{b}_2)}_\QQ. 
	\end{split}
	\]

	We used the inductive hypothesis in the last row of each case. Since $\Qp{\dot{b}_1 = \dot{b}_2} = \Qp{\dot{b}_1 \subseteq \dot{b}_2} \wedge \Qp{\dot{b}_2 \subseteq \dot{b}_1}$, the proof for $\psi$ atomic is complete.
	
	For $\psi$ quantifier-free formula the proof is immediate since $i$ is an embedding hence preserves $\vee$, $\neg$. Suppose now that $\psi = \exists x \in y ~ \phi$ is a $\Delta_0$ formula.
	\[
	\begin{split}
	    i&\cp{\Qp{\exists x \in \dot{b}_1 \phi(x,\dot{b}_1,\ldots,\dot{b}_n)}_\BB} \\
		& = \bigvee \bp{  i\cp{\dot{b}_1(\dot{a})} \wedge i\cp{\Qp{\phi(\dot{a},\dot{b}_1,\ldots,\dot{b}_n)}_\BB}: {\dot{a} \in \dom(\dot{b}_1)}} \\
		& = \bigvee \bp{ i(\dot{b}_1(\dot{a})) \wedge \Qp{\phi\cp{\hat{\imath}(\dot{a}),\hat{\imath}(\dot{b}_1),\ldots,\hat{\imath}(\dot{b}_n)}}_\QQ : {\dot{a} \in \dom(\dot{b}_1)}} \\
		& = \Qp{\exists x \in \hat{\imath}(\dot{b}_1) ~ \phi\cp{x,\hat{\imath}(\dot{b}_1),\ldots,\hat{\imath}(\dot{b}_n)}}_\QQ
	\end{split}
	\]
	

	Furthermore, if $\psi = \exists x ~ \phi$ is a $\Sigma_1$ formula, by the fullness lemma there exists a $\dot{a} \in V^\BB$ such that $\Qp{\exists x \phi(x,\dot{b}_1,\ldots,\dot{b}_n)}_\BB = \Qp{\phi(\dot{a},\dot{b}_1,\ldots,\dot{b}_n)}_\BB$ hence
	\[
	\begin{split}
		i\cp{\Qp{\exists x \phi(x,\dot{b}_1,\ldots,\dot{b}_n)}_\BB} &= i\cp{\Qp{\phi(\dot{a},\dot{b}_1,\ldots,\dot{b}_n)}_\BB} \\
		&= \Qp{\phi\cp{\hat{\imath}(\dot{a}),\hat{\imath}(\dot{b}_1),\ldots,\hat{\imath}(\dot{b}_n)}}_\QQ \\
		&\leq \Qp{\exists x\phi\cp{x,\hat{\imath}(\dot{b}_1),\ldots,\hat{\imath}(\dot{b}_n)}}_\QQ
	\end{split}
	\]
	Thus, if $\phi$ is a $\Delta_1$ formula, either $\phi$ and $\neg \phi$ are $\Sigma_1$ hence the above inequality holds and also
	\[
	\begin{split}
		i\cp{\Qp{\phi(\dot{b}_1,\ldots,\dot{b}_n)}_\BB} &= \neg i\cp{\Qp{\neg\phi(\dot{b}_1,\ldots,\dot{b}_n)}_\BB} \\
		&\geq \neg \Qp{\neg \phi\cp{\hat{\imath}(\dot{b}_1),\ldots,\hat{\imath}(\dot{b}_n)}}_\QQ \\
		&= \Qp{\phi\cp{\hat{\imath}(\dot{b}_1),\ldots,\hat{\imath}(\dot{b}_n)}}_\QQ, \\
	\end{split}
	\]
	concluding the proof.
\end{proof}

\begin{notation}
In general all over these notes for the sake of readability we shall confuse $\BB$-names with their defining properties.
Recurring examples of this behavior are the following:
\begin{itemize}
\item
If we have in $V$ a collection $\{\dot{b}_i:i\in I\}$ of $\BB$-names, we confuse
$\{\dot{b}_i:i\in I\}$ with a $\BB$-name $\dot{b}$ such that for all $\dot{a}\in V^{\BB}$
\[
\Qp{\dot{a}\in\dot{b}}=\Qp{\exists i\in \check{I}\dot{a}=\dot{b}_i}.
\]
\item
If $i:\BB\to\QQ$ is a complete homomorphism, we denote by
$\QQ/i[\dot{G}]$ a $\BB$-name $\dot{b}$ such that
\[
\Qp{\dot{b}
\mbox{ is the quotient of $\QQ$ modulo the ideal generated by the dual of }i[\dot{G}_{\BB}]}=\1_{\BB}.
\]
\end{itemize}
\end{notation}

%% file: 3_Iterations.tex
\section{Iteration systems}\label{sec:iterations}

In this section we will present iteration systems and some of their algebraic properties. We refer to later sections an analysis of their forcing properties. In order to develop the theory of iterations, from now on we shall consider only regular embeddings.

\subsection{Definitions and basic properties}
\begin{definition}
	$\FFF=\{i_{\alpha \beta}:\BB_\alpha\to \BB_\beta:\alpha \leq \beta < \lambda\}$ is a \emph{complete iteration system} of complete boolean algebras iff for all $\alpha \leq \beta \leq \gamma < \lambda$:
	\begin{enumerate}
		\item $\BB_\alpha$ is a complete boolean algebra and $i_{\alpha \alpha}$ is the identity on it;
		\item $i_{\alpha \beta}$ is a regular embedding with associated retraction $\pi_{\alpha \beta}$;
		\item $i_{\beta \gamma}\circ i_{\alpha \beta}=i_{\alpha \gamma}$.
 \end{enumerate}

If $\gamma < \lambda$, we define $\FFF \res \gamma = \bp{i_{\alpha \beta}: \alpha \leq \beta < \gamma}$.
\end{definition}

\begin{definition}
Let $\FFF$ be a complete iteration system of length $\lambda$. Then:

\begin{itemize}
\item
 The \emph{inverse limit} of the iteration is
	\[
	T(\FFF) = \bp{ f \in \prod_{\alpha < \lambda} \BB_\alpha : ~ \forall \alpha \forall \beta > \alpha ~ \pi_{\alpha \beta}(f(\beta)) = f(\alpha) }
	\]
	and its elements are called \emph{threads}. 
\item
	The \emph{direct limit} is
	\[
	C(\FFF) = \bp{ f \in T(\FFF) : ~ \exists \alpha \forall \beta > \alpha ~ f(\beta) = i_{\alpha\beta}(f(\alpha)) }
	\]
	and its elements are called \emph{constant threads}. 
	The support of a constant thread $\supp(f)$ is 
	the least $\alpha$ such that $i_{\alpha\beta}\circ f(\alpha)=f(\beta)$ for all $\beta\geq\alpha$. 
\item
	The \emph{revised countable support limit} is
	\[
	RCS(\FFF) = \bp{ f \in T(\FFF) : ~ f \in C(\FFF) \vee \exists \alpha ~ f(\alpha) \Vdash \cf(\check{\lambda}) = \check{\omega} }
	\]
\end{itemize}	
\end{definition}

We can define on $T(\FFF)$ a natural join operation.

\begin{definition}
	Let $A$ be any subset of $T(\FFF)$. We define the \emph{pointwise supremum} of $A$ as
	\[
	\psup A=\ap{ \bigvee\{f(\alpha):f\in A\}:\alpha<\lambda }.
	\]
\end{definition}

The previous definition makes sense since by Proposition \ref{eRetrProp}.\ref{eRPJoins} $\psup A$ is a thread.

\begin{definition}
	Let $\FFF = \{i_{\alpha \beta} : \alpha \leq \beta < \lambda\}$ be an iteration system. 
	For all $\alpha < \lambda$, we define $i_{\alpha\lambda}$ as
	\[
	\begin{array}{llll}
		i_{\alpha\lambda} : & \BB_\alpha &\to& C(\FFF) \\
		&b& \mapsto & \ap{\pi_{\beta,\alpha}(b) : ~ \beta < \alpha}^\smallfrown \ap{i_{\alpha\beta}(b) : ~ \alpha \leq \beta < \lambda }
	\end{array}
	\]
	and $\pi_{\alpha \lambda}$ 
	\[
	\begin{array}{llll}
		\pi_{\alpha\lambda} : & T(\FFF) &\to& \BB_\alpha \\
		&f& \mapsto & f(\alpha)
	\end{array}
	\]
	When it is clear from the context, we will denote $i_{\alpha\lambda}$ by $i_\alpha$ and $\pi_{\alpha\lambda}$ by $\pi_\alpha$.
\end{definition}

	\begin{fact}
    We may observe that:

	\begin{enumerate}
	
		\item
			$C(\FFF)\subseteq RCS(\FFF)\subseteq T(\FFF)$ are partial orders with the order relation given by pointwise comparison of threads.
					
		\item 
			Every thread in $T(\FFF)$ is completely determined by its tail. Moreover every thread in $C(\FFF)$ is entirely determined by the restriction to its support. Hence, given a thread $f \in T(\FFF)$, for every $\alpha < \lambda$ $f \res \alpha$ determines a constant thread $f_\alpha \in C(\FFF)$ such that $f \leq_{T(\FFF)} f_\alpha$. 
		
		\item
			It follows that for every $\alpha < \beta < \lambda$, $i_{\alpha \lambda} = i_{\alpha \beta} \circ i_{\beta \lambda}$.
		
		\item
			$i_{\alpha \lambda}$ can naturally be seen as a regular embedding of
			$\BB_\alpha$ in any of $\RO(C(\FFF))$, $\RO(T(\FFF))$, $\RO(RCS(\FFF))$. Moreover
			by Property \ref{eRetrProp}.\ref{eRPJoins} in all three cases
			$\pi_{\alpha \lambda}=\pi_{i_{\alpha,\lambda}}\res P$ 
			where $P=C(\FFF), T(\FFF), RCS(\FFF)$.

		\item
			If $\FFF$ is an iteration of length $\lambda$, and $g: \cf(\lambda) \to \lambda$ is an increasing cofinal map, then we have the followings isomorphisms of partial orders:
			\[C(\FFF)\cong C(\{i_{g(\alpha) g(\beta)} : \alpha \leq \beta < \cf(\lambda)\});\]
			\[T(\FFF)\cong T(\{i_{g(\alpha) g(\beta)} : \alpha \leq \beta < \cf(\lambda)\});\]
			\[RCS(\FFF)\cong RCS(\{i_{g(\alpha) g(\beta)} : \alpha \leq \beta < \cf(\lambda)\}).\]
			hence we will always assume w.l.o.g. that $\lambda$ is a regular cardinal.
	\end{enumerate}
	\end{fact}
	
	\begin{remark}
	It must be noted that if $A$ is an infinite subset of $T(\FFF)$, $\psup A$ might \emph{not} be the least upper bound of  $A$ in $\RO(T(\FFF))$, as shown in Example \ref{xSups}. A sufficient condition on $A$ for this to happen is given by Lemma~\ref{iTFSup}
	below.
	\end{remark}
	
	\begin{definition}	
	$C(\FFF)$ inherits the structure of a boolean algebra with boolean operations defined as follows:
		\begin{itemize}
			\item $f\wedge g$ is the unique thread $h$ whose support $\beta$ is the max of the support of $f$ and $g$ and is such that $h(\beta)=f(\beta)\wedge g(\beta)$,
			\item $\neg f$ is the unique thread $h$ whose support $\beta$ is the support of $f$ such that $h(\beta)=\neg f(\beta)$.
		\end{itemize}
	\end{definition}
	
	\begin{fact}
	\begin{enumerate}
		\item If $g\in T(\FFF)$ and $h\in C(\FFF)$ we can check that $g\wedge h$, defined as the thread where eventually all coordinates $\alpha$ are the pointwise meet of $g(\alpha)$ and $h(\alpha)$, is the infimum of $g$ and $h$ in $T(\FFF)$.
		\item There can be nonetheless two distinct incompatible threads 
		$f,g\in T(\FFF)$ such that $f(\alpha)\wedge g(\alpha)>\0_{\BB_\alpha}$ for all 
		$\alpha<\lambda$. Thus in general the pointwise meet of two threads is not even a thread, as shown in Example \ref{xWedge}.
	\end{enumerate}
    \end{fact}

	\begin{remark}
	In general $C(\FFF)$ is not complete and $RO(C(\FFF))$ cannot be identified with a complete subalgebra of $\RO(T(\FFF))$ (i.e. $C(\FFF)$ and $T(\FFF)$ as forcing notions in general share little in common), as shown in Example \ref{xSubstruct}. However, $\RO(C(\FFF))$ can be identified with a subalgebra of $T(\FFF)$ that is complete (even though it is not a complete subalgebra), as shown in the following proposition.
    \end{remark}

\begin{proposition} \label{iCFinTF}
	Let $\FFF = \bp{i_{\alpha \beta} : \alpha \leq \beta < \lambda}$ be an iteration system. Then $\RO(C(\FFF)) \simeq D = \bp{ f \in T(\FFF): ~ f = \psup \bp{g\in C(\FFF):~ g \leq f} }$.
\end{proposition}

\begin{proof}
	The isomorphism associates to a regular open $U \in \RO(C(\FFF))$ the thread $k(U) = \psup  U$, with inverse $k^{-1}(f) = \bp{g \in C(\FFF):~ g \leq f}$.

	First, we prove that $k^{-1} \circ k(U) = \bp{ g \in C(\FFF): ~ g \leq \psup U } = U$. Since $\psup U > \bigvee U$, it follows that $U \subseteq k^{-1} \circ k(U)$. Furthermore, since $U$ is a regular open set, if $g \notin U$, there exists a $g' \leq g$ that is in the interior of the complement of $U$ (i.e., $\forall g'' \leq g' ~ g'' \notin U$). So suppose towards a contradiction that there exist a $g \leq \tilde{\bigvee} U$ as above (i.e., $\forall g' \leq g ~ g' \notin U$). Let $\alpha$ be the support of $g$, so that $g(\alpha) \leq \bigvee \bp{f(\alpha): ~ f \in U}$. Then, there exists an $f \in U$ such that $f(\alpha)$ is compatible with $g(\alpha)$, hence $f \wedge g > \0$ and is in $U$ (since $U$ is open). Since $f \wedge g \leq g$, this is a contradiction.

	It follows that $k(U) \in D$ for every $U \in \RO(C(\FFF))$.  Moreover, $k^{-1}(f)$ is in $\RO(C(\FFF))$ (i.e., is regular open). In fact, it is open and if $g \notin k^{-1}(f)$ then $g \nleq f$ and this is witnessed by some $\alpha > \supp(g)$, so that $g(\alpha) \nleq f(\alpha)$. Let $h = i_\alpha(g(\alpha)\setminus f(\alpha)) > \0$, then for all $h' \leq h$, $h'(\alpha) \perp f(\alpha)$ hence $h' \nleq f$, thus $k^{-1}(f)$ is regular.

	Furthermore, $k^{-1}$ is the inverse map of $k$ since we already verified that $k^{-1} \circ k(U) = U$ and for all $f \in D$, $k \circ k^{-1}(f) = f$ by definition of $D$. Finally, $k$ and $k^{-1}$ are order-preserving maps since $U_1 \subseteq U_2$ iff $\psup U_1 \leq \psup U_2$.
	\end{proof}

As noted before, the notion of supremum in $T(\FFF)$ may not coincide with the notion of pointwise supremum. However, in some cases it does, for example:

\begin{lemma}\label{iTFSup}
	Let $\FFF = \{i_{\alpha \beta} : \alpha \leq \beta < \lambda\}$ be an iteration system and $A \subseteq T(\FFF)$ be an antichain such that $\pi_{\alpha \lambda}[A]$ is an antichain for some $\alpha < \lambda$. Then $\psup  A$ is the supremum of the elements of $A$ in $\RO(T(\FFF))$.
\end{lemma}

\begin{proof}
	Suppose by contradiction that $\bigvee A < \psup  A$ in $\RO(T(\FFF))$. Then there exists $g \in T(\FFF)$ such that $0< g \leq \neg \bigvee A \wedge \psup A$.  Let $\alpha < \lambda$ be such that $\pi_{\alpha \lambda} [A] $ is an antichain and let $f \in A$ be such that $f(\alpha)$ is compatible with $g(\alpha)$. Such $f$ exists because $g(\alpha) \leq \bigvee\bp{ f(\alpha): {f \in A}}$ so, since $g(\alpha)\neq \0$, there exists $f \in A$, $f(\alpha) \parallel g(\alpha)$.  We are going to prove that $g$ and $f$ are compatible.
	Consider 
	\[
 		 h = \ap{ g(\beta) \wedge i_{\alpha, \beta} \circ f(\alpha) : ~ \alpha \leq \beta < \lambda }.
	\] 
	Then $h \leq g$ and it is a thread of $T(\FFF)$.
	In fact since $i_{\alpha,\beta}=i_{\gamma, \beta} \circ i_{\alpha, \gamma}$ for each $\alpha \leq \gamma \leq \beta < \lambda$
	\[ 
  		\pi_{\gamma, \beta}(h(\beta))= \pi_{\gamma, \beta} (g(\beta) \wedge i_{\alpha,\beta}\circ f(\alpha)) = \pi_{\gamma, \beta}(g(\beta)) \wedge i_{\alpha, \gamma}(f(\alpha))=h(\gamma).
	\]
	It only remains to prove that $h(\beta) \leq f(\beta)$ for each $\beta \geq \alpha$. We have $h(\beta) \leq g(\beta) \leq \sup\{t(\beta): t \in A\}$ and also $h(\beta)$ is incompatible with $t(\beta)$ for all $f \neq t \in A$.
	In fact $h(\alpha)= g(\alpha) \wedge f(\alpha) \leq f(\alpha) \perp t(\alpha);$ now suppose by contradiction that $g(\beta) \wedge i_{\alpha, \beta}f(\alpha) \parallel t(\beta),$ so there exists $r$ such that $r \leq g(\beta) \wedge i_{\alpha,\beta}(f(\alpha))$ and $r \leq t(\beta)$, then we obtain a contradiction:
	\[
	  \pi_{\alpha,\beta} (r) \leq g(\alpha) \wedge i_{\alpha,\gamma}(f(\alpha)) \mbox{ and } \pi_{\alpha,\beta} (r) \leq t(\alpha).
	\]
	Thus
	\[
	  h(\beta) \leq \bigvee\{t(\beta) : t \in A\} \wedge \cp{\neg \bigvee\{t(\beta): t \in A, t \neq f\}}= f(\beta)
	\]
	for all $\beta \geq \alpha$. So $g$ and $f$ are compatible. Contradiction.
\end{proof}

\subsection{Sufficient conditions for $C(\FFF) = T(\FFF)$}

Even though in general $C(\FFF)$ is different from $T(\FFF)$, in certain cases 
they happen to coincide:

\begin{lemma} \label{iBaumCorollary}
	Let $\FFF = \{i_{\alpha \beta} : \alpha \leq \beta < \lambda\}$ be an iteration system such that $C(\FFF)$ is ${<}\lambda$-cc. Then $T(\FFF) = C(\FFF)$ is a complete boolean algebra.
\end{lemma}

\begin{proof}
	First, since every element of $\RO(C(\FFF))$ is the supremum of an antichain in $C(\FFF)$, since $C(\FFF)$ is ${<}\lambda$-cc and since $\lambda$ is regular, the supremum of such an antichain can be computed in some $\BB_\alpha$ for $\alpha < \lambda$ hence $\RO(C(\FFF)) = C(\FFF)$.

	Let $f$ be in $T(\FFF)\setminus C(\FFF)$. 
	Since $f$ is a non-constant thread, for all $\alpha < \beta$ we have that 
	$i_{\alpha\beta}(f(\alpha)) \geq f(\beta)$ and for all $\alpha$ there is an ordinal
	$\beta_\alpha$ such that 
	$i_{\alpha\beta_\alpha}(f(\alpha)) > f(\beta_\alpha)$. By restricting to a subset of $\lambda$
	w.l.o.g. we can suppose that $f(\beta)<i_{\alpha\beta}(f(\alpha))$ for all $\beta>\alpha$.
	Hence $\{i_{\alpha\lambda}(f(\alpha)):\alpha<\lambda\}$ is a strictly descending sequence of length $\lambda$ of elements in $C(\FFF)^+$. 
	From a descending sequence we can always define an antichain in $C(\FFF)$ setting $a_\alpha = i_{\alpha\lambda}(f(\alpha)) \wedge \neg i_{\alpha+1,\lambda}(f(\alpha+1))$. Since $C(\FFF)$ is ${<}\lambda$-cc, this antichain has to be of size less than $\lambda$ hence for coboundedly many $\alpha$, $a_\alpha = \0$ hence $f(\alpha+1) = i_{\alpha,\alpha+1}(f(\alpha))$, so $f \in C(\FFF)$, contradiction.
\end{proof}

\begin{theorem}[Baumgartner] \label{iBaumgartner}
	Let $\FFF = \{i_{\alpha \beta} : \alpha \leq \beta < \lambda\}$ be an iteration system such that $\BB_\alpha$ is ${<}\lambda$-cc for all $\alpha$ and $S = \bp{\alpha: ~ \BB_\alpha \cong \RO(C(\FFF \res \alpha))}$ is stationary. Then $C(\FFF)$ is ${<}\lambda$-cc.
\end{theorem}

\begin{proof}
	Suppose by contradiction that there exists an antichain $\ap{f_\alpha: \alpha < \lambda}$.
	Let $h: \lambda \to \lambda$ be such that $h(\alpha) > \alpha, \supp(f_\alpha)$. Let $C$ be the club of closure points of $h$ (i.e. such that for all $\alpha \in C$, $h[\alpha] \subseteq h(\alpha)$). Then we can define a regressive function
	\[
	\begin{array}{llll}
		\phi : & S &\to& \lambda \\
		&\alpha &\mapsto& \min\bp{\supp(g): ~ g \in C(\FFF \res \alpha)^+, g < f_\alpha(\alpha)}
	\end{array}
	\]
	and a corresponding function $\psi: S \to C(\FFF)$ such that $\supp(\psi(\alpha)) = \phi(\alpha)$, $\psi(\alpha)(\alpha) < f_\alpha(\alpha)$. By Fodor's Lemma let $\xi \in \lambda$, $T \subset S$ be stationary  such that $\phi[T] =\bp{\xi}$.

	Since $\psi[T \cap C]$ has size $\lambda$ and $\BB_\xi$ is ${<}\lambda$-cc, there are $\alpha,\beta \in T \cap C$ such that $\psi(\alpha) \wedge \psi(\beta) \geq b > \0$ for some $b$ with $\supp(b) = \xi$. Moreover $f_\alpha(\alpha) > b(\alpha)$ holds and the support of $b$ is below $\alpha$, so that $f_\alpha \wedge b > \0$. Furthermore, $f_\beta(\beta) > b(\beta) \geq \cp{f_\alpha \wedge b}(\beta)$ and the support of $f_\alpha \wedge b$ is below $\beta$, thus $f_\beta \wedge f_\alpha \wedge b > \0$ contradicting the hypothesis that $\ap{f_\alpha: \alpha < \lambda}$ is an antichain.
\end{proof}

%% file: 4_Quotients.tex
\section{Two-step iterations and generic quotients}~\label{sec:twostepitquot}

In the first part of this section we define two-step iteration $\mathbb{B}*\dot{\mathbb{Q}}$,
following Jech~\cite[chapter 16]{JECH}
and we study the basic properties of the natural 
regular embedding of $\BB$ into $\mathbb{B}*\dot{\mathbb{Q}}$ where $\dot{\mathbb{Q}}$
is a $\BB$-name for a complete boolean algebra.

In the second part of this section we study the properties of generic quotients
given by $\BB$-names $\QQ/i[\dot{G}_{\BB}]$ where $i:\BB\to\QQ$ is a complete homomorphism
and we show that if we have a commutative diagram of complete homomorphisms:
\[
	\begin{tikzpicture}[xscale=1.5,yscale=-1.2]
		\node (B) at (0, 0) {$\BB$};
		\node (Q0) at (1, 0) {$\QQ_0$};
		\node (Q1) at (1, 1) {$\QQ_1$};
		\path (B) edge [->]node [auto] {$\scriptstyle{i_0}$} (Q0);
		\path (B) edge [->]node [auto,swap] {$\scriptstyle{i_1}$} (Q1);
		\path (Q0) edge [->]node [auto] {$\scriptstyle{j}$} (Q1);
	\end{tikzpicture}
	\]
and $G$ is a $V$-generic filter for $\BB$, then the map defined by 
$j/G([c]_{i_0[G]})=[j(c)]_{i_1[G]}$ is a complete
homomorphism in $V[G]$ and we also show a converse of this property.

In the third part of this section we show that the two approaches are equivalent in the sense that
$i:\BB\to\QQ$ is a complete homomorphism iff $\QQ$ is isomorphic to $\BB*\QQ/i[\dot{G}_{\BB}]$
and we prove a converse of the above factorization property when we start from $\BB$-names
for regular embeddings $\dot{k}:\dot{\QQ}\to\dot{\RR}$.

Finally in the last part we apply the above results to analyze generic quotients of iteration systems.

\subsection{Two-step iterations}

We present two-step iterations following~\cite{JECH}.

\begin{definition}
	Let $\BB$ be a complete boolean algebra, and $\dot{\QQ}$ be a $\BB$-name for a complete boolean algebra. We denote by $\BB \ast \dot{\QQ}$ the boolean algebra defined in $V$ whose
	elements are the equivalence classes of $\BB$-names for elements of $\dot{\QQ}$
	(i.e. $\dot{a}\in V^{\BB}$ such that $\Qp{\dot{a}\in\dot{\QQ}}_{\BB}=\1_{\BB}$) 
	modulo the equivalence relation:
	\[
	\dot{a} \approx \dot{b} ~ \Leftrightarrow ~ \Qp{\dot{a} = \dot{b}}_\BB = \1,
	\]
	with the following operations:
	\[ 
		[\dot{d}] \vee_{\BB \ast \dot{\QQ}} [\dot{e}] = [\dot{f}] \iff \Qp{\dot{d} \vee_{\dot{\QQ}} \dot{e} =\dot{f}}=\1_{\BB};
	\]
	
	\[
		\qquad \neg_{\BB \ast \dot{\QQ}}[\dot{d}]=[\dot{e}] 
	\]
	for any $\dot{e}$ such that $\Qp{\dot{e}=\neg_{\dot{\QQ}}\dot{d}}=\1_{\BB}$.
\end{definition}

Literally speaking our definition of $\BB\ast\dot{\QQ}$ yields an object whose domain is a family of 
proper classes of $\BB$-names. By means of Scott's trick we can arrange so that 
$\BB\ast\dot{\QQ}$ is indeed a set. We leave the details to the reader.

\begin{lemma}\label{lem:twostepit}
	Let $\BB$ be a complete boolean algebra, and $\dot{\QQ}$ be a $\BB$-name for a complete boolean algebra. Then $\BB\ast\dot{\QQ}$ is a complete boolean algebra and the maps 
	$i_{\BB\ast\dot{\QQ}}$, $\pi_{\BB\ast\dot{\QQ}}$ defined as
	\[
	\begin{array}{cccl}
		i_{\BB\ast\dot{\QQ}} :& \BB &\to& \BB \ast \dot{\QQ} \\
		& b &\mapsto& [\dot{d}_b]_\approx \\
		\pi_{\BB\ast\dot{\QQ}} :& \BB \ast \dot{\QQ} &\to& \BB \\
		& \qp{\dot{c}}_\approx &\mapsto& \Qp{\dot{c} > \0}
	\end{array}
	\]
	where $\dot{d}_b\in V^{\BB}$ is a $\BB$-name for an element of $\dot{\QQ}$ such that  
	$\Qp{ \dot{d}_b =\1_{\dot{\QQ}} }_{\BB} = b$ and 
	$\Qp{ \dot{d}_b =\0_{\dot{\QQ}} }_{\BB} = \neg b$,
	are a regular embedding with its associated retraction. 
\end{lemma}

\begin{proof}
	We leave to the reader to verify that $\BB * \dot{\QQ}$ is a boolean algebra.
	We can also check that
	\[
	[\dot{c}]\leq[\dot{a}]\iff \Qp{\dot{c}\vee\dot{a}=\dot{a}}=\1_{\BB}
	\iff \Qp{\dot{c}\leq\dot{a}}=\1_{\BB}.
	\]
	Observe that $\BB * \dot{\QQ}$ is also complete: if 
	$\{ [\dot{d}_\alpha] : \alpha < \delta \} \subseteq \BB\ast\dot{\QQ}$, 
	let $\dot{c}$ be such that 
	$\Qp{\dot{c}=\bigvee\bp{\dot{d}_\xi: {\xi<\delta}}}=\1$.
	Then $[\dot{c}] \geq \bigvee\bp{ [\dot{d}_\xi]: {\xi<\delta}}$
	since for all $\alpha<\delta$
	\[
	\Qp{\bigvee\bp{ \dot{d}_\xi: {\xi<\delta}}\geq\dot{d}_\alpha}=\1_{\BB}.
	\]
	Moreover 
	if
	\[
		\Qp{\dot{a} \geq \dot{d}_\alpha}=\1,
	\] 
	for all $\alpha<\delta$, then
	\[
	\bigwedge\bp{\Qp{\dot{a} \geq \dot{d}_\alpha}:{\alpha<\delta}}=\1,
	\]
	
	thus $\Qp{\dot{a} \geq \dot{c}}=\1$, hence $[\dot{a}] \geq [\dot{c}]$, which gives that
	$[\dot{c}] = \bigvee\bp{[\dot{d}_\alpha]: {\alpha<\delta} }$. 
	
	Now we prove that $i_{\BB\ast\dot{\QQ}}$ is a regular embedding and that $\pi_{\BB\ast\dot{\QQ}}$ is its associated retraction.
		\begin{itemize}
    		\item First of all a standard application of the mixing lemma to the maximal antichain
		$\{b,\neg b\}$ and the family of $\BB$-names $\{\dot{1},\dot{0}\}$
		shows that for each $b \in \BB$ there exists a unique 
		$[\dot{d}_b] \in   		\BB\ast\dot{\QQ}$ 
		such that $\Qp{ \dot{d}_b =1 } = b$ and $\Qp{ \dot{d}_b =0} = \neg b$. 
		Now, suppose that there exist $[\dot{c}], [\dot{d}] \in \BB\ast\dot{\QQ}$ such that 
		$i_{\BB\ast\dot{\QQ}}(b) =[\dot{c}]=[\dot{d}]$. Then
			\[
				\begin{split}
   				 	b&=\Qp{\dot{c}=1}\wedge \Qp{\dot{d}=1}\leq \Qp{\dot{c}=\dot{d}}\\
    				\neg b&= \Qp{\dot{c}=0}\wedge \Qp{\dot{d}=0}\leq \Qp{\dot{c}=\dot{d}}.
  			    \end{split}
			\]
			Hence $\1 = b \vee \neg b \leq \Qp{\dot{c}=\dot{d}}$ and this implies $[\dot{c}]=[\dot{d}]$.
	
			\item $i_{\BB\ast\dot{\QQ}}$ preserves negation. Observe that $\neg [\dot{d}_{b}]= [\dot{d}_{\neg b}]$. In  fact we have that 
			\[
			    \neg \Qp{\dot{d}_{b} = 1}=  \Qp{\dot{d}_{b} = 0}=\Qp{\neg \dot{d}_{b} = 1}= \neg b
			\]  
			    and similarly 
			\[ 
				\neg \Qp{\dot{d}_{b} = 0}= \Qp{\neg \dot{d}_{b} = 0}=b;
			\]
			so, thanks to the uniqueness proved above, $[\neg \dot{d}_{b}]=[\dot{d}_{\neg b}]$. Therefore
			\[
			   i_{\BB\ast\dot{\QQ}}(\neg b)= [\dot{d}_{\neg b}]= \neg[\dot{d}_{b}]= \neg i_{\BB\ast\dot{\QQ}}(b).
			\]
			
			\item $i_{\BB\ast\dot{\QQ}}$ preserves joins. Consider $\{b_\alpha \in \BB : \alpha < \delta\}$. We have that
			\[
				\Qp{\bigvee \dot{d}_{b_\alpha} = 0} =  \bigwedge \Qp{\dot{d}_{b_\alpha}=0}= \bigwedge (\neg b_\alpha) = \neg (\bigvee b_\alpha).
			\]
			We have also 
			\[\begin{split}
			     \Qp{\bigvee\dot{d}_{b_\alpha} = 1} &\leq \Qp{\bigvee \dot{d}_{b_\alpha} > 0} = \bigvee \Qp{\dot{d}_{b_\alpha} > 0} \\ &=\bigvee \Qp{\dot{d}_{b_\alpha} = 1} \leq  \Qp{\bigvee\dot{d}_{b_\alpha} = 1};
			  \end{split}
			\]
			then $\Qp{\bigvee\dot{d}_{b_\alpha} = 1} =\bigvee \Qp{\dot{d}_{b_\alpha} = 1}= \bigvee b_\alpha$. Hence
			\[
			    i_{\BB\ast\dot{\QQ}}\cp{\bigvee b_\alpha}= \qp{\bigvee\dot{d}_{b_\alpha}}= \bigvee \qp{ i_{\BB\ast\dot{\QQ}} (b_\alpha)}.
			\]
			
			\item $i_{\BB\ast\dot{\QQ}}$ is regular. If $i_{\BB\ast\dot{\QQ}}(b)=i_{\BB\ast\dot{\QQ}}(b') = [\dot{d}]$, then $b'= \Qp{\dot{d}=1}= b$.
			
			\item We have to show that $\pi_{i_{\BB\ast\dot{\QQ}}}([\dot{c}]) = \Qp{\dot{c} >0}$: 
			by applying the definition of retraction associated to $i_{\BB\ast\dot{\QQ}}$,
			\[
				\pi_{i_{\BB\ast\dot{\QQ}}}([\dot{c}])= \bigwedge\{b \in \BB : i_{\BB\ast\dot{\QQ}}(b)\geq [\dot{c}]\}.
			\]
			If $b$ is such that $i_{\BB\ast\dot{\QQ}}(b)\geq [\dot{c}]$, then $\Qp{\dot{d_b} \geq \dot{c}}=\1$ and we obtain
			\[
			   b= \Qp{\dot{d}_b = 1}=\Qp{\dot{d}_b >0} \geq \Qp{\dot{c} >0}\wedge \Qp{\dot{d_b} \geq \dot{c}}= \Qp{\dot{c} >0},
			\]
			so we have the first inequality
			\[
			\pi_{i_{\BB\ast\dot{\QQ}}}([\dot{c}])\geq\Qp{\dot{c}>0}.
			\]
			In order to obtain the other one, let $i_{\BB\ast\dot{\QQ}}(\Qp{\dot{c} >0})=[\dot{d}]$.
			Then on the one hand:
			\[
				\neg \Qp{\dot{c} = 0} = \Qp{\dot{c} > 0} = \Qp{\dot{d} = 1} \leq \Qp{\dot{c} \leq \dot{d}}.
			\]
			On the other hand:
			\[
			 \Qp{\dot{c} = 0} \leq \Qp{\dot{c} \leq \dot{d}}.
			\]
			In particular since $\neg \Qp{\dot{c} = 0}\vee\Qp{\dot{c} = 0}=\1_{\BB}$ we get that
			\[
			\Qp{\dot{c} \leq \dot{d}}=\1_{\BB},
			\]
			and thus that
			$[\dot{c}]\leq [\dot{d}] = i_{\BB\ast\dot{\QQ}} (\Qp{\dot{c} >0})$, 
			i.e. 
			\[
			\pi_{i_{\BB\ast\dot{\QQ}}}([\dot{c}])\leq\Qp{\dot{c}>0}
			\]
as was to be shown.
			
	 \end{itemize}
\end{proof}

When clear from the context, we shall feel free to omit the subscripts in $i_{\BB\ast\dot{\QQ}}$, $\pi_{\BB\ast\dot{\QQ}}$.

\begin{remark}
	This definition is provably equivalent to Kunen's two-step iteration of posets (as in \cite{KUNEN}), i.e. 
	$\RO(P \ast \dot{Q})$ is isomorphic to $\RO(P) \ast \RO(\dot{Q})$.
\end{remark}

We shall need in several occasions the following fact:

\begin{fact}\label{6antichainDQ}
	$A= \{[\dot{c}_\alpha]_{\approx} : \alpha \in \lambda\}$ is a maximal antichain in $\DD= \BB * \dot{\QQ}$, if and only if  
	\[
		  \Qp{ \{\dot{c}_\alpha : \alpha \in \lambda\} \mbox{ is a maximal antichain in } \dot{\QQ} } = \1.
	\]
\end{fact}

\begin{proof}
	It is sufficient to observe the following:
	\[
 		 \Qp{ \dot{c}_\alpha \wedge \dot{c}_\beta = \dot{0} } = \1 \iff \qp{\dot{c}_\alpha}_\approx \wedge \qp{\dot{c}_\beta}_\approx = \qp{\dot{0}}_\approx;
	\]
	\[
 		 \Qp{ \bigvee \dot{c}_{\alpha} = \dot{1} } = \1 \iff \bigvee\qp{\dot{c}_\alpha}_\approx= \qp{\bigvee \dot{c}_\alpha}_\approx=\qp{\dot{1}}_\approx.
	\]
\end{proof}

We want also to address briefly how to handle the case of three steps iteration in our framework.
\begin{fact}
Assume $\BB\in V$ is a complete boolean algebra, $\dot{\QQ}\in V^{\BB}$ is a
$\BB$-name for a complete boolean algebra and $\dot{\RR}\in V^{\BB\ast\dot{\QQ}}$ is 
a $\BB*\dot{\QQ}$-name for a complete boolean algebra.

Let $G$ be any ultrafilter on $\BB$ and $K$ be an ultrafilter on $\BB/G$.
Set
\[
H=\{c:[c]_G\in K\}
\]
Then 
\[
K=\{[c]_G: c\in H\}
\]
and
$((\BB\ast\dot{\QQ})*\dot{\RR}/_G)/_K$ is isomorphic to 
$(\BB\ast\dot{\QQ})*\dot{\RR}/_H$ via the map $[[c]_G]_K\mapsto [c]_H$.
\end{fact}

\subsection{Generic quotients}

We now outline the definition and properties of generic quotients.

\begin{proposition}\label{qQuotientCba}
	Let $i: \BB \to \QQ$ be a regular embedding of complete boolean algebras and $G$ be a $V$-generic filter for $\BB$. Then $\QQ /_G$, defined with abuse of notation as the quotient of $\QQ$ with the filter generated by $i[G]$, is a boolean algebra in $V[G]$.
\end{proposition}

\begin{proof}
	We have that
	\[
		V[G] \models \QQ \mbox{ is a boolean algebra and } i[G] \mbox{ generates a filter on } \QQ.
	\]
	Thus $\QQ /_G$ is a boolean algebra in $V[G]$ such that
	\begin{itemize}
	\item $\qp{a}=\qp{b}$ if and only if $a \triangle b \in {i[G]}^*$;
	\item $\qp{a} \vee \qp{b} = \qp{a \vee b}$;
	\item $\neg \qp{a} = \qp{\neg a}$;
	\end{itemize}
	where ${i[G]}^*$ is the dual ideal of the filter ${i[G]}$.
\end{proof}

\begin{lemma} \label{qElmRepr}
	Let $i: \BB \to \QQ$ be a regular embedding, $\dot{G}$ be the canonical name for a generic filter for $\BB$ and $\dot{d}$ be a $\BB$-name for an element of $\QQ/_{\dot{G}}$. Then there exists a unique $c \in \QQ$ such that $\Qp{\dot{d} = \qp{c}_{i[\dot{G}]} } = \1_\BB$.
\end{lemma}

\begin{proof}
	First, notice that the $\BB$-name for the dual of the filter generated by $i[\dot{G}]$ is 
	$\dot{I} = \bp{ \ap{c,\neg\pi_i(c)}: c \in \QQ }$.
	\begin{description}
		\item[Uniqueness.] Suppose that $c_0, c_1$ are such that $\Qp{\dot{d} = \qp{c_k}_{\dot{I}} } = \1_\BB$ for $k < 2$. Then $\Qp{\qp{c_0}_{\dot{I}} = \qp{c_1}_{\dot{I}} } = \1_\BB$ hence $\Qp{ c_0 \triangle c_1 \in \dot{I} } = \neg \pi_i(c_0 \triangle c_1) = \1_\BB$. This implies that $\pi_i(c_0 \triangle c_1) = \0_\BB \Rightarrow c_0 \triangle c_1 = \0_\BB \Rightarrow c_0 = c_1$.
		\item[Existence.] Let $A \subset \BB$ be a maximal antichain deciding the value of $\dot{d}$, and for every $a \in A$ let $c_a$ be such that $a \Vdash \dot{d} = \qp{c_a}_{\dot{I}}$. Let $c \in \QQ$ be such that $c = \bigvee \bp{ i(a) \wedge c_a : ~ a \in A}$, so that
		\[
		\Qp{\qp{c}_{\dot{I}} = \qp{c_a}_{\dot{I}}} = \Qp{c \triangle c_a \in \dot{I}} = \neg \pi_i(c \triangle c_a) \geq \neg \pi_i(i(\neg a)) = a
		\]
		since $c \triangle c_a \leq \neg i(a) = i(\neg(a))$. Thus,
		\[
		\begin{array}{lll}
			\Qp{\dot{d} = \qp{c}_{\dot{I}} } &\geq& \Qp{\dot{d} = \qp{c_a}_{\dot{I}}} \wedge \Qp{\qp{c}_{\dot{I}} = \qp{c_a}_{\dot{I}}} \geq a \wedge a = a
		\end{array}
		\]
		The above inequality holds for any $a \in A$, so $\Qp{\dot{d} = \qp{c}_{\dot{I}} } \geq \bigvee A = \1_\BB$ concluding the proof. \qedhere
	\end{description}
\end{proof}

\begin{proposition}\label{qQuotientC}
	Let $i: \BB \to \QQ$ be a regular embedding of complete boolean algebras and $G$ be a $V$-generic filter for $\BB$. Then $\QQ /_G$ is a complete boolean algebra in $V[G]$.
\end{proposition}

\begin{proof}
	By Proposition \ref{qQuotientCba}, we need only to prove that $\QQ/_G$ is complete. Let $\{\dot{c}_\alpha : \alpha < \delta\}\in V$ be a set of $\BB$ names for elements of $\QQ/_{\dot{G}}$. Then, by 
	Lemma \ref{qElmRepr}, for each $\alpha < \delta$ there exists $d_{\alpha} \in \QQ$ such that 
	\[
		\Qp{\dot{c}_\alpha = \qp{d_\alpha}_{i[\dot{G}]} } = \1_{\BB}.
	\]
	We have that $\bigvee d_{\alpha} \in \QQ$, since $\QQ$ is complete. Let $c \in \QQ$ be such that  $V[G]\models \forall \alpha < \delta \qp{c} \geq \qp{d_\alpha}$, then
	\[
		\neg \pi (d_\alpha \wedge \neg c)
 = \Qp{ d_\alpha \wedge \neg c \in {i[\dot{G}]}^*} = \Qp{\qp{c} \geq \qp{d_\alpha} }\in G	
 \]
	So $\pi(d_\alpha \wedge \neg c) \not\in G $ for all $\alpha<\delta$.
	In particular since $\{\pi(d_\alpha \wedge \neg c):\alpha<\delta\}\in V$ is disjoint from $G$, we also have that 
	\[
	d=\bigvee\{\pi(d_\alpha \wedge \neg c):\alpha<\delta\}=
	\pi(\neg c\wedge\bigvee\{d_\alpha:\alpha<\delta\})\not\in G.
	\]
	
	This gives that if $\pi(c)\in G$ then 
	\[
		V[G] \models \qp{c} \geq \qp{\bigvee{d_\alpha}}, 
	\]
	while if $\pi(c)\not\in G$, then $\pi(\neg c)\in G$ and thus
	\[
		\begin{split}
		\bigvee\{\pi(d_\alpha):\alpha<\delta\}&=
		\bigvee\{\pi(d_\alpha \wedge \neg c):\alpha<\delta\}\vee
		\bigvee\{\pi(d_\alpha \wedge c):\alpha<\delta\}\\
		&\leq d\vee \pi(c)\not\in G,
		\end{split}
	\]
	in which case $\qp{d_\alpha}$ and $\qp{\bigvee\bp{{d_\alpha}:{\alpha<\delta}}}$ are all
	equal to $\0_{\QQ/G}$.
	In either cases $ \qp{\bigvee\{d_\alpha:\alpha<\delta\}}$ is the least upper bound of the family
	$\{\qp{d_\alpha}:\alpha<\delta\}$ in $V[G]$.
	This shows that $V[G] \models \QQ/_G$ is complete for all $V$-generic filters $G$.
\end{proof}

The construction of generic quotients can be defined also for regular embeddings:

\begin{proposition}\label{lem:firstfctlem}
	Let $\BB$, $\QQ_0$, $\QQ_1$ be complete boolean algebras, and let $G$ be a $V$-generic filter for $\BB$. Let $i_0$, $i_1$, $j$ form a commutative diagram of regular embeddings as in the following picture:
	\[
	\begin{tikzpicture}[xscale=1.5,yscale=-1.2]
		\node (B) at (0, 0) {$\BB$};
		\node (Q0) at (1, 0) {$\QQ_0$};
		\node (Q1) at (1, 1) {$\QQ_1$};
		\path (B) edge [->]node [auto] {$\scriptstyle{i_0}$} (Q0);
		\path (B) edge [->]node [auto,swap] {$\scriptstyle{i_1}$} (Q1);
		\path (Q0) edge [->]node [auto] {$\scriptstyle{j}$} (Q1);
	\end{tikzpicture}
	\]
	Then $j/_G: \QQ_0/_G \to \QQ_1/_G$ defined by $j/_G( \qp{c}_{i_0[G]} ) = \qp{ j(c) }_{i_1[G]}$ is a well-defined regular embedding of complete boolean algebras in $V[G]$ with associated retraction $\pi$ such that $\pi( \qp{c}_{i_1[G]} ) = \qp{ \pi_j(c) }_{i_0[G]}$.
\end{proposition}
\begin{proof}
	By Proposition \ref{qQuotientC}, $j/_G$ is a map between complete boolean algebras.
	\begin{itemize}
	\item $j/_G$ is well defined: If $\qp{c}_{i_0[G]}=\qp{d}_{i_0[G]}$, then $c \triangle d \in {i_0[G]}^*$. Hence $j(c) \triangle j(d) = j(c \triangle d) \in {i_1[G]}^*$. So $\qp{j(c)}_{i_1[G]} = \qp{j(d)}_{i_1[G]}$.
	\item $j/_G$ is a complete homomorphism of boolean algebras:
	\[
		\begin{split}
			j/_G(\neg\qp{c_\alpha}_{i_0[G]})&=j/_G(\qp{\neg c_\alpha}_{i_0[G]})= \qp{j(\neg c_\alpha)}_{i_1[G]}\\
			&= \qp{\neg j( c_\alpha)}_{i_1[G]}=\neg \qp{j( c_\alpha)}_{i_1[G]}.
		\end{split}
	\] 
	Moreover, by Proposition \ref{qQuotientC},
	\[
		\begin{split}
			j/_G\cp{\bigvee\qp{c_\alpha}_{i_0[G]}}&=j/_G\cp{\qp{\bigvee c_\alpha}_{i_0[G]}}= \qp{j\cp{\bigvee c_\alpha}}_{i_1[G]}\\
			&= \qp{\bigvee j( c_\alpha)}_{i_1[G]}=\bigvee \qp{j( c_\alpha)}_{i_1[G]}.
		\end{split}
	\] 
	\item $j/_G$ is injective: Let $c,d \in \QQ_0$ be such that $ j/_G(\qp{c}_{i_0[G]})=j/_G(\qp{d}_{i_0[G]})$, then $j(c \triangle d) \in {i_1[G]}^*$. So there exists $g \not\in G$ such that $j(c \triangle d) \leq i_1(g) = j(i_0(g))$; since $j$ is injective, then $c \triangle d \in i_0[G]^*$. 
	
	\item $\pi( \qp{c}_{i_1[G]} ) = \qp{ \pi_j(c) }_{i_0[G]}$: 
	\[
		V[G]\models \pi(\qp{c}_{i_1[G]})= \bigwedge \{\qp{b}_{i_0[G]} \in \QQ_0/_G : j/_G([b]_{i_0[G]}) \geq \qp{c}_{i_1[G]}\}.
	\]
	
	  
	Now observe that for any $b\in \QQ_0$:
	\[
		 \Qp{ j/_{\dot{G}}(\qp{b}_{i_0[\dot{G}]}) \geq \qp{c}_{i_1[\dot{G}]}}= \Qp{ c \wedge \neg j(b) \in {i_1[\dot{G}]}^* }= \neg \pi_{i_1}(c \wedge \neg j(b)).
	\]
	
	Thus $j(b)\geq c$ iff $c\wedge \neg j(b)=\0_{\QQ_1}$ iff $\pi_{i_1}(c \wedge \neg j(b))=\0_{\BB}$
	iff 
	\[ 	
	\Qp{ j/_{\dot{G}}(\qp{b}_{i_0[\dot{G}]}) \geq \qp{c}_{i_1[\dot{G}]}}=\1_{\BB}.
	\]
	
	Now $ \pi_j(c \wedge \neg j(b))=\pi_j(c)\wedge\neg b$ and 
	$\pi_{i_0}(\pi_j(c)\wedge\neg b)=\pi_{i_1}(c \wedge \neg j(b))$.
	
	Thus 
	\[
		V[G]\models  j/_G([b]_{i_0[G]}) \geq \qp{c}_{i_1[G]}
	\]
	iff $\pi_j(c)\wedge\neg b\in i_0[G]^*$. 
	
	Given such a $b$, let $b'=b\vee(\pi_j(c)\wedge\neg b)$.
	Then 
	\[
	\neg\pi_{i_0} (\pi_j(c)\wedge\neg b)\leq \Qp{ [b]_{i_0[G]}=[b']_{i_0[G]}}
	\]  
	and thus
	$\Qp{ [b]_{i_0[G]}=[b']_{i_0[G]}}\in G$ and:
	\[
	\begin{split}
	\Qp{ j/_{\dot{G}}([b']_{i_0[G]})\geq [c]_{i_1[G]}}&=\neg \pi_{i_1}(c \wedge \neg j(b')) \\
	&=\neg \pi_{i_0}(\pi_j(c) \wedge \neg b').
	\end{split}
	\]

	Now observe that
	
	\begin{align*}
	\pi_j(c) \wedge \neg b'&= \pi_j(c) \wedge \neg (b\vee(\pi_j(c)\wedge\neg b)) \\
	&=\pi_j(c) \wedge \neg b\wedge \neg(\pi_j(c)\wedge\neg b)\\
	&= \pi_j(c) \wedge \neg b\wedge (\neg\pi_j(c)\vee b) \\
	&=(\pi_j(c) \wedge \neg b\wedge \neg\pi_j(c))\vee (\pi_j(c) \wedge \neg b\wedge  b)=\0_{\QQ_0}.
	\end{align*}

Thus 
\[
\Qp{j/_{\dot{G}}([b']_{i_0[G]}) \geq [c]_{i_1[G]}}=\neg \pi_{i_0}(\pi_j(c) \wedge \neg b')=\neg\pi_{i_0}(\0_{\QQ_0})=\1_{\BB},
\]
and $[b]_{i_0[G]^*}=[b']_{i_0[G]^*}$.

This gives that
	\[
		\begin{split}
		V[G]\models \pi(\qp{c}_{i_1[G]})&= \bigwedge \{\qp{b}_{i_0[G]} \in \QQ_0/_G : j(b) \geq c\}\\
		&=\qp{\bigwedge \{b \in \QQ_0 : j(b) \geq c\}}_{i_0[G]}=[\pi_j(c)]_{i_0[G]}
		\end{split}
	\]
as was to be shown.
%
	\end{itemize}
\end{proof}

\subsection{Equivalence of two-step iterations and regular embeddings}
We are now ready to prove that two-step iteration and regular embedding capture the same concept.
\begin{theorem}\label{qIso}
	If $i: \BB \rightarrow \QQ$ is a regular embedding of complete boolean algebra, then 
	$\BB * \QQ/_{i[\dot{G}_{\BB}]} \cong \QQ$.
\end{theorem}

\begin{proof}
	Let 
		\[
			\begin{split}
			i^* : \QQ &\rightarrow \BB* \QQ/_{\dot{G}} \\
			c &\mapsto \qp{[c]_{i[\dot{G}]}}_\approx.
			\end{split}
		\]
	$i^*$ is a regular embedding by Proposition~\ref{qQuotientC} $\QQ/_{\dot{G}}$ and by definition of two-step iteration; in fact:
		\[
			i^*(\neg c)= \qp{[\neg c]_{i[\dot{G}]}}_\approx = \qp{\neg [c]_{i[\dot{G}]}}_\approx= \neg \qp{ [c]_{i[\dot{G}]}}_\approx = \neg i^*(c);
		\]
	and
		\[
			i^*(\bigvee c_\alpha)= \qp{[\bigvee c_\alpha]_{i[\dot{G}]}}_\approx = \qp{\bigvee[c_\alpha]_{i[\dot{G}]}}_\approx= \bigvee \qp{ [c_\alpha]_{i[\dot{G}]}}_\approx = \bigvee i^*(c_\alpha).
		\]

	Moreover it is a bijection since, by Lemma \ref{qElmRepr}, for all $\dot{d}$ $\BB$-name for an element in $\QQ/_{\dot{G}}$, there exists a unique $c \in \QQ$ such that $\Qp{[c]_{i[\dot{G}]}= \dot{d} }=\1$, and since, by  definition of two-step iteration  $[\dot{d}_1]_\approx = [\dot{d}_2]_\approx$ iff  $\Qp{\dot{d}_1=\dot{d}_2}=\1$.
\end{proof}


\begin{proposition}\label{prop:embfromembnames}
	Let $\dot{\QQ}_0$, $\dot{\QQ}_1$ be $\BB$-names for complete boolean algebras, and let $\dot{k}$ be a $\BB$ name for a regular embedding from $\dot{\QQ}_0$ to $\dot{\QQ}_1$. Then 
	there is a regular embedding $i : ~ \BB \ast \dot{\QQ}_0 \to \BB \ast \dot{\QQ}_1$ such that
	\[
	\Qp{\dot{k}=i/\dot{G}_{\BB}}=\1_{\BB}.
	\]
\end{proposition}

\begin{proof}
	Let 
	\[
		\begin{split}
		i: \BB \ast \dot{\QQ}_0 &\to \BB \ast \dot{\QQ}_1\\
		[\dot{d}]_\approx &\mapsto [\dot{k} (\dot{d})]_\approx.
		\end{split}
	\]
	
	Since $\dot{k}$  is a $\BB$-name
	for a regular embedding with boolean value $\1_{\BB}$, we have that 
	\[  \begin{split}
			\qp{\dot{d}}_\approx = \qp{\dot{e}}_\approx \iff \Qp{\dot{d}= \dot{e}}=\1 \iff\\
	   		\Qp{\dot{k} (\dot{d})= \dot{k} (\dot{e})} =\1 \iff  
			\qp{\dot{k} (\dot{d})}_\approx = \qp{\dot{k} (\dot{e})}_\approx 
		\end{split}	
	\]
	This shows that $i$ is well defined and injective. We have that
	$i$ is a complete homomorphism, since
	\[ 
		\begin{split}
		&i\cp{\neg \qp{\dot{c}}_\approx} =  i\cp{ \qp{\neg \dot{c}}_\approx}=\qp{\dot{k}\cp{\neg \dot{c}}}_\approx \\
		&=\qp{\neg \dot{k}\cp{ \dot{c}}}_\approx =\neg\qp{ \dot{k}\cp{ \dot{c}}}_\approx =\neg i\cp{\qp{\dot{c}}_\approx};
		\end{split}
	\]
	and
	\[
		\begin{split}
		&i\cp{\bigvee \qp{\dot{c}_\alpha}_\approx} =  i\cp{ \qp{\bigvee \dot{c}_\alpha}_\approx}=\qp{\dot{k}\cp{\bigvee \dot{c}_\alpha}}_\approx \\
		&=\qp{\bigvee \dot{k}\cp{ \dot{c}_\alpha}}_\approx =\bigvee\qp{ \dot{k}\cp{ \dot{c}_\alpha}}_\approx =\bigvee i\cp{ \qp{\dot{c}_\alpha}_\approx}. 
		\end{split}
	\]
	
	Moreover if $G$ is $V$-generic for $\BB$, $\dot{k}_G=i/_G$. As a matter of fact, thanks to the diagram
	
	\[
		\begin{tikzpicture}[xscale=1.5,yscale=-1.2]
			\node (B) at (0, 0) {$\BB$};
			\node (Q0) at (1.25, 0) {$\BB\ast\dot{\QQ}_0$};
			\node (Q1) at (1.25, 1.25) {$\BB\ast\dot{\QQ}_1$};
			\path (B) edge [->]node [auto] {$\scriptstyle{i_{\BB\ast \dot{\QQ}_0}}$} (Q0);
			\path (B) edge [->]node [auto,swap] {$\scriptstyle{i_{\BB\ast \dot{\QQ}_1}}$} (Q1);
			\path (Q0) edge [->]node [auto] {$\scriptstyle{i}$} (Q1);
		\end{tikzpicture}
	\]
	
	\[
	i/_{G} (\qp{\qp{\dot{c}}_\approx}_{i_{\BB\ast\dot{\QQ}_0}[G]})= \qp{i\cp{\qp{\dot{c}}_\approx}}_{i_{\BB\ast\dot{\QQ}_1}[G]}
	= \qp{\qp{\dot{k}\cp{\dot{c}}}_\approx}_{i_{\BB\ast\dot{\QQ}_1}[G]}.
	\]
	\end{proof}

\subsection{Generic quotients of iteration systems}
The results on generic quotients of the previous sections generalize without much effort to
iteration systems. In the following we outline how this occurs.

\begin{lemma}
	Let $\FFF=\{i_{\alpha \beta}:\BB_\alpha \to \BB_\beta:\alpha \leq \beta < \lambda\}$ be a complete iteration system of complete boolean algebras, $G_\gamma$ be a $V$-generic filter for $\BB_\gamma$.
	Then $\FFF/_{G_\gamma}=\{i_{\alpha \beta}/_{G_\gamma}:\gamma < \alpha \leq \beta < \lambda\}$ is a complete iteration system in $V[G_\gamma]$.
\end{lemma}


\begin{lemma} \label{qThreadRepr}
	Let $\FFF=\{i_{\alpha \beta}:\BB_\alpha \to \BB_\beta:\alpha \leq \beta < \lambda\}$ be a complete iteration system of complete boolean algebras, $\dot{G}_\alpha$ be the canonical name for a generic filter for $\BB_\alpha$ and $\dot{f}$ be a $\BB_\alpha$-name for an element of $T(\FFF/_{\dot{G}_\alpha})$. Then there exists a unique $g \in T(\FFF)$ such that $\Qp{\dot{f} = \qp{g}_{\dot{G}_\alpha} } = \1_{\BB_\alpha}$.
\end{lemma}

\begin{proof}
	We proceed applying Lemma \ref{qElmRepr} at every stage $\beta > \alpha$.
	\begin{description}
		\item[Existence.] For every $\beta > \alpha$, by hypothesis $\dot{f}(\beta)$ is a name for an element of the quotient $\BB_\beta/_{i_{\alpha\beta}[\dot{G}_\alpha]}$. Let $g(\beta)$ be the unique element of $\BB_\beta$ such that $\Qp{\dot{f}(\beta) = \qp{g(\beta)}_{i_{\alpha\beta}[\dot{G}_\alpha]} } = \1_{\BB_\alpha}$. Then,
		\[
		\begin{split}
			\Qp{\dot{f} = \qp{g}_{\dot{G}_\alpha}} &= \Qp{\forall \beta \in \lambda ~ \dot{f}(\beta) = \qp{g(\beta)}_{\dot{G}_\alpha}} \\
			&= \bigwedge\bp{ \Qp{\dot{f}(\beta) = \qp{g(\beta)}_{i_{\alpha\beta}[\dot{G}_\alpha]} } :{\beta \in \lambda}}= \bigwedge \1_{\BB_\alpha} = \1_{\BB_\alpha}
		\end{split}
		\]
		\item[Uniqueness.] If $g'$ is such that $\Qp{\dot{f} = \qp{g'}_{\dot{G}_\alpha} } = \1_{\BB_\alpha}$ then for every $\beta > \alpha$, $\Qp{\dot{f}(\beta) = \qp{g'(\beta)}_{i_{\alpha\beta}[\dot{G}_\alpha]} } = \1_{\BB_\alpha}$. Such an element is unique by Lemma \ref{qElmRepr}, hence $g'(\beta) = g(\beta)$ defined above, completing the proof. \qedhere
	\end{description}
\end{proof}

\begin{remark}
	All the results in this section can be generalized to complete homomorphisms $i$, by considering $i \res \coker(i)$ that is a regular embedding as already noted in Definition \ref{eCoker}.
\end{remark}

%% file: 5_Examples.tex
\section{Examples and counterexamples}\label{sec:examples}

In this section we shall examine some aspects of iterated systems by means of examples. In the first one we will see that $T(\FFF)$ may not be a complete boolean algebra, and that $C(\FFF)$ and $T(\FFF)$ as forcing notions share little in common.
In the second one we show that the pointwise meet of two threads may not even be a thread.
In the third one, we will justify the introduction of
$\RCS$-limits showing that in many cases $C(\FFF)$ collapses $\omega_1$ even if all factors of the iteration are preserving $\omega_1$. This shows that in order to produce a limit of an iteration system that preserves $\omega_1$ one needs to devise subtler notions of limits than 
full and direct limits.
This motivates the results of sections~\ref{sec:semiproper} and~\ref{sec:semiproperiter} where it is shown that $\RCS$-limits are a nice notion of limit, since
$\RCS$-iterations of semiproper posets are semiproper and preserve $\omega_1$.
The last iteration system shall provide also an example of iteration in which the direct limit is taken stationarily often but $T(\FFF) \neq \RO(C(\FFF))$.

\subsection{Distinction between direct limits and full limits}

\begin{example} \label{xSups} \label{xSubstruct}
Let $\FFF_0= \{i_{n,m}: \BB_n \rightarrow \BB_m : n < m < \omega\}$ be an iteration system such that for all $n \in \omega$ $1_{\BB_n} \Vdash \BB_{n+1}/_{\dot{G}} \neq 2$, and $\vp{\BB_0}$ is atomless and infinite. 
\end{example}

\begin{lemma} \label{xLemmaEx1}
	There exists $t_m \in C(\FFF_0)$ for each $m \in \omega$ such that the followings hold:
	\begin{enumerate}
	\item $\{t_{n+1} : n \in \omega\}$ is an antichain;
	\item $\tilde{\bigvee}\bp{t_{n+1}: n \in \omega}=1$;
	\item there exists $t \in T(\FFF_0)$ such that for all $n\in \omega$, $t \perp t_{n+1}$. 
	\end{enumerate}  	
\end{lemma}

\begin{proof}

	Since $\1_{\BB_n} \Vdash \BB_{n+1}/_{\dot{G}} \neq 2$, there exists $\dot{a}_{n+1} \in V^{\BB_n}$ such that $\1_{\BB_n} \Vdash 0 < \dot{a}_{n+1} < 1$. Then let $a_{n+1} \in \BB_{n+1}$ be such that $\Qp{\dot{a}_{n+1}=[a_{n+1}]_{\dot{G}}}=\1_{\BB_n}$, which exists by Lemma \ref{qElmRepr}. Then $\pi_{n,n+1}(a_{n+1})=\1$ and $\pi_{n,n+1}(\neg a_{n+1})=\1$. Let $a_0 = \1$.

	For all $n >0$, let
	\[
		t_n = \ap{i_{n,m} (\neg a_n) \wedge \bigwedge\bp{ i_{l,m}(a_l): {l<n}}: m \in \omega, m > n}.
	\]
    First of all we have
    		\[ 
    	     	\begin{split}
                 &\pi_{n,n+1}(\bigwedge\bp{i_{l,n+1}(a_l):{l \leq n+1}}) \\ &=\pi_{n,n+1}(i_{n,n+1}(\bigwedge\bp{i_{l,n}(a_l): {l < n+1}}) \wedge a_{n+1}) \\
                 &= \bigwedge\bp{ i_{l,n}(a_l):{l \leq n}} \wedge \pi_{n,n+1}(a_{n+1})= \bigwedge\bp{ i_{l,n}(a_l):{l \leq n}}.
    	    	\end{split}
    		\]
    This implies also that  
    	    \[
    	       	    \begin{split}
    	       	    &t_{n+1}(n)= \pi_{n,n+1}( \neg a_{n+1} \wedge \bigwedge\bp{ i_{l,n+1}(a_l):{l < n+1}})=\\
    	       	     &= \pi_{n,n+1}(\neg a_{n+1}) \wedge \pi_{n,n+1}(i_{n,n+1} \bigwedge\bp{ i_{l,n}(a_l):{l \leq n}} )= \bigwedge\bp{ i_{l,n}(a_l):{l \leq n}}.
    	        	\end{split}
    		\]
	\begin{enumerate} 
		\item Observe that for all $0< m < n \in \omega$ $t_n \perp t_m$. As a matter of fact
			 \[
				\begin{split}
					t_m(n)&= i_{m,n} (\neg a_m) \wedge \bigwedge\bp{ i_{l,n}(a_l):{l<m}} < \neg i_{m,n}(a_m),\\
					t_n(n)&= \neg a_n \wedge \bigwedge\bp{ i_{l,n}(a_l):{l<n}} < i_{m,n}(a_m).
				\end{split}
			 \]
		\item In order to prove $\tilde{\bigvee}\bp{t_m: 0<m\in \omega} = 1$, we prove by induction on $n$ that 
			\[
				\bigvee\bp {t_m(n): {0 < m \leq n+1}} = \1.
			\] 
	    	If $n=0$ then $t_1(0)=\pi_{0,1} (a_1 \wedge i_{0,1}(a_0))=\pi_{0,1}(a_1)=\1$. Now assume that it holds for $n$.
			Observe that
				\[ 	
					\begin{split}
					&t_{n+1}(n+1) \vee t_{n+2}(n+1)= (\neg a_{n+1} \wedge \bigwedge\bp{ i_{m,n+1}(a_m):{m < n+1}}) \\
					&\vee ( a_{n+1} \wedge \bigwedge\bp{i_{m,n+1}(a_m):{m < n+1}}= \bigwedge\bp{ i_{m,n+1}(a_m):{m < n+1}}\\
					&= i_{n,n+1}(\bigwedge\bp{ i_{m,n} (a_m):{m \leq n}} )= i_{n,n+1}(t_{n+1}(n))\\
					\end{split}
				\]
			Then  
				\[	
					\begin{split}
					&i_{n,n+1} (\bigvee\bp{ t_m(n): {0< m \leq n}}) \vee t_{n+1}(n+1) \vee t_{n+2}(n+1)= \\
					&= i_{n,n+1} (\bigvee\bp{ t_m(n):{0<m \leq n}}) \vee i_{n,n+1}(t_{n+1}(n))= \\
					&= i_{n,n+1}(\bigvee \bp{t_m(n): {0< m \leq n+1}}))=\1.
					\end{split}
				\]
		
		\item Let $t= \ap{ \bigwedge\bp{ i_{m,n}(a_m):{m\leq n}} : n \in \omega}$. It is a thread since, thanks to the first point, for all $l < n$: 
				\[
					\pi_{l,n} (\bigwedge\bp{ i_{m,n}(a_m):{m \leq n}}) = \bigwedge\bp{ i_{m,l}(a_m):{m \leq l }}.
				\]
		      Moreover we have that $t \perp t_{n}$ for all $n\in \omega\setminus\{0\}$, since $t_n(n) < \neg a_n$ and $t(n+1) < i_{n,n} (a_n)= a_n$.	
		\end{enumerate} 
\end{proof}

\begin{proposition}
$\RO(C(\FFF_0))$ is not a complete subalgebra of $T(\FFF_0)$. Moreover $T(\FFF_0)$ is not closed under suprema.
\end{proposition}
\begin{proof}
    We have that for each $m \in \omega$ there exists $t_m \in C(\FFF_0)$ as in Lemma \ref{xLemmaEx1}.
	Hence $\bigvee \bp{t_{n+1}:n \in \omega} \neq \tilde{\bigvee}\bp{t_{n+1}:n \in \omega}=1$. Since $t_{n+1} \in C(\FFF_0)$ for all $n \in \omega$, this implies also that $\RO(C(\FFF_0))$ is not a complete subalgebra of $T(\FFF_0)$.
	Moreover since it is easy to check that a thread $t\in T(\FFF_0)$ is a majorant of a family $A$ of threads
	in $T(\FFF_0)$ iff $t\geq \tilde{\bigvee}A$, we also get that $T(\FFF_0)$ is not closed under suprema of its subfamilies and thus cannot be a complete boolean algebra.
\end{proof}

\subsection{The pointwise meet of threads may not be a thread}

Let $\FFF_0$ be the iteration system defined in example~\ref{xSups}.

\begin{proposition} \label{xWedge}
	There exist $f, g \in T(\FFF_0)$ such that $f \perp g$ in $T(\FFF)$ but $f(n) \wedge g(n) > \0$ for all $n < \omega$.
\end{proposition}

\begin{proof}
	Let $\ap{a_n: n < \omega}$ be a descending sequence in $\BB_0$ such that $\bigwedge \ap{a_n: n< \omega} = \0$ (it exists since it can be defined from a maximal antichain of $\BB_0$ of countable size). Let $d_n \in \BB_n$ be such that $\pi(d_n) = \pi(\neg d_n) = \1$ as in the previous subsection. Let $b_n = d_n \vee i_{0n}(a_n)$, $c_n = \neg d_n \vee i_{0n}(a_n)$, so that $b_n \wedge c_n = i_{0n}(a_n)$ and $\pi_{n-1,n}(b_n) = \pi_{n-1,n}(c_n) = \1$.
	
	As in the previous subsection $f = \bigwedge\bp{ i_n(b_n):n\in \omega}$, $g = \bigwedge\bp{ i_n(c_n):n \in \omega}$ are threads in $T(\FFF)$, such that $f(n) \wedge g(n) = \bigwedge\bp{ i_{0n}(a_m):{m \leq n}} = i_{0n}(a_n) > \0$ since $\ap{a_n: n<\omega}$ is a descending sequence.
	
	Furthermore, suppose by contradiction that there exist a non-zero thread $h \leq f,g$. Then for all $n < \omega$, $h(n) \leq f(n) \wedge g(n) = i_{0n}(a_n)$ and $h(0) = \pi_{0n}(h(n)) \leq \pi_{0n}\circ i_{0n}(a_n) = a_n$ for all $n$. Thus, $h(0) \leq \bigwedge\bp{ a_n: n \in \omega} = \0$, a contradiction.
\end{proof}

\subsection{Direct limits may not preserve $\omega_1$}

To develop an example of a case where the direct limit of an iteration system of length bigger than $\omega_1$ does not preserve $\omega_1$ we shall use the following forcing notion.

\begin{definition}
Let $\lambda$ be a regular cardinal. 
	\emph{Namba forcing} $\Nm(\lambda)$ is the poset of all perfect trees $T \subseteq \lambda^{<\omega}$ (i.e. everbranching and such that for every $t \in T$, the set $\{\alpha < \lambda: ~ t^\smallfrown \alpha \in T \}$ has cardinality either $1$ or $\lambda$), ordered by reverse inclusion.
\end{definition}

In the example below we shall use only the following well-known properties of Namba forcing:

\begin{fact}
	$\Nm(\lambda)$ is stationary set preserving (and thus preserves $\omega_1$)
	 and forces the cofinality of $\lambda$ to become $\omega$ and its size to become $\omega_1$.
\end{fact}

%

\begin{example} \label{xCollapse} \label{xStatDirLim}
	Let $\FFF_1 = \{i_{\alpha,\beta}: \BB_\alpha \rightarrow \BB_\beta : \alpha \leq \beta < \lambda\}$ be an iteration system such that $S = \bp{\alpha < \lambda: ~ \BB_\alpha = C(\FFF_1 \res \alpha)}$ is stationary, and suppose that $\BB_0$ is the boolean completion of the Namba forcing $\Nm(\lambda)$ and $\BB_{\alpha+1} /_{\dot{G}_\alpha}$ is forced to have antichains of uncountable size.
\end{example}

\begin{proposition}
	$\RO(C(\FFF_1))$ is a proper subset of $T(\FFF_1)$ and collapses $\omega_1$.
\end{proposition}

\begin{proof}
	Let $\dot{f} \in V^{\BB_0}$ be the canonical name for a cofinal function from $\omega$ to $\lambda$, and let $\dot{A}_\alpha$ be a name for an antichain of size $\omega_1$ in $\BB_{\alpha+1}/_{\dot{G}_\alpha}$, with $A_\alpha = \{a^\alpha_\beta : ~ \beta < \omega_1\}$ the corresponding antichain of size $\omega_1$ in $\BB_{\alpha+1}$ obtained by repeated application of Lemma \ref{qElmRepr}. Then for all $\alpha$, $\beta$ we have that $\Qp{\0 < \qp{a^\alpha_\beta}_{\dot{G}_\alpha} < \1} = \1_{\BB_\alpha}$ hence $\pi_{\alpha,\alpha+1}(a^\alpha_\beta) = \pi_{\alpha,\alpha+1}(\neg a^\alpha_\beta) = \{\1_{\BB_\alpha}\}$.

	Let $\dot{t} \in V^{\BB_0}$ be a name for the thread in $T(\FFF_1/_{\dot{G}_0})$ defined by
	the requirement that for all $n<\omega$
	\[
	\Qp{\dot{t}(\dot{f}(n)) = [a^{\dot{f}(n)}_0]_{\dot{G}_0}}_{\BB_0}=\1_{\BB_0}.
	\]
	Let $t \in T(\FFF_1)$ be the canonical representative for $\dot{t}$ obtained from Lemma \ref{qThreadRepr}. Suppose by contradiction that  $t \in \RO(C(\FFF_1))^+$, so that there exists an $r \leq t$ in $C(\FFF_1)^+$. Since $\dot{f}$ is a $\BB_0$-name for a cofinal increasing function from 
	$\omega$ to $\lambda$, $r$ cannot decide in $C(\FFF_1)$ a bound for the value of $\dot{f}(n)$ for cofinally many $n$, else $\lambda$ would have countable cofinality in $V$. Let $b \in \BB_0$, $b \leq r(0)$ be such that 
	\[
	b \Vdash_{\BB_0} \dot{f}(n) = \gamma
	\] 
	with $\gamma > \supp(r)$ and $n$ large enough so that $r$ cannot bound the value of $\dot{f}(n)$.
	 Let
	$i_\alpha:\BB_\alpha\to C(\FFF_1)$ be the canonical embedding of $\BB_\alpha$ into $C(\FFF_1)$.
	Then $t \wedge i_0(b) \leq i_\gamma(a^\gamma_0)$ but $r \wedge i_0(b)$ cannot be below $i_\gamma(a^\gamma_0)$ since it has support smaller than $\gamma$ (and so is compatible con $\neg a^\gamma_0$, that is an element that projects to $\1_{\BB_\gamma}$), a contradiction which shows that
	$t\not\in \RO(C(\FFF_1))$.

	For the second part of the thesis, let $G_0$ be $V$-generic for $\BB_0$, $f=\dot{f}_{G_0}$.
	Let $i_\alpha/_{G_0}$ denote the canonical embedding of $\BB_\alpha/_{G_0}$ into
	$C(\FFF_1/_{G_0})$.
	
	Define $\dot{g}$ to be a $C(\FFF_1/_{G_0})$-name in $V[G_0]$ for a function from $\omega$ to $\omega_1$ as follows:
	\[
	\Qp{\dot{g}(n) = \check{\beta}} = i_{f(n)}/_{G_0}([a^{f(n)}_\beta]_{G_0}).
	\]
	Then $\dot{g}$ is forced to be a $C(\FFF_1/_{G_0})$-name for a surjective
	map from $\omega$ to $\omega_1$, since for every $t \in C(\FFF_1/_{G_0})$ and $\beta \in \omega_1$ we can find an $n$ such that $f(n) > \supp(t)$ so that 
	\[
	[t]_{G_0} \wedge i_{f(n)}/_{G_0}([a^{f(n)}_\beta]_{G_0})
	\] 
	is positive and forces $\beta$ to be in the range of $\dot{g}$. Thus, $C(\FFF_1/_{G_0})$ collapses $\omega_1$ to $\omega$ for every $G_0$ $V$-generic for $\BB_0$. Since $C(\FFF_1) = \BB_0 \ast C(\FFF_1/_{\dot{G}_0})$ the same holds for $C(\FFF_1)$, as witnessed by the following 
	$C(\FFF_1)$-name $\dot{h}$ for a function
	\[
	\Qp{\dot{h}(\check{n}) = \check{\beta}}_{\RO(C(\FFF_1))} = \bigvee\bp{ i_0\cp{\Qp{\dot{f}(\check{n}) = \check{\alpha}}_{\BB_0}}  \wedge a^\alpha_\beta :{\alpha \in \lambda}},
	\]
	completing the proof.
\end{proof}

%% file: 6_Semiproperness.tex
\section{Semiproperness} \label{sec:semiproper}

In this section we shall introduce the definition of semiproperness and some 
equivalent formulations of it.  
Our final aim is to show that $\RCS$-limits of semiproper posets 
yield a semiproper poset (this will be achieved in the next section).
It is rather straightforward to check that semiproper forcings preserve $\omega_1$ as 
well as the stationarity of ground model subsets of $\omega_1$, thus $\RCS$-limits are particularly 
appealing in order to prove consistency results over $H_{\omega_2}$ and are actually 
the tool to obtain the 
consistency of
strong forcing axioms. This consistency result (i.e. the proof of the consistency of 
Martin's maximum relative to a supercompact cardinal) will be the content of 
the last section of these notes.
Most of all our considerations about semiproperness transfer without much effort to 
properness with the obvious changes in the definitions. 
However since we decided to focus our analysis on semiproper posets we shall leave to 
the interested reader to transfer our result to the case of proper forcings.

\subsection{Algebraic definition of properness and semiproperness}

To state an algebraic formulation of semiproperness, we first need the following definition.

\begin{definition}\label{def:semipropcba}
	Let $\BB$ be a complete boolean algebra, $M \prec H_\theta$ for some $\theta \gg \vp{\BB}$, $\PD(\BB)$ be the collection of predense subsets of $\BB$ of size at most $\omega_1$. 
	The boolean value
	\[
	\sg(\BB,M) = \bigwedge\bp{  \bigvee (D \cap M):{D \in \PD(\BB) \cap M} }
	\]
	is the \emph{degree of semigenericity} of $M$ with respect to $\BB$. 
\end{definition}


The next results show that the degree of semigenericity can be also calculated from maximal antichains, and behaves well with respect to the restriction operation.

\begin{proposition}\label{PredenseAntichain}
	Let $\BB$, $M$, $\PD(\BB)$ be as in the previous definition, and let $\A(\BB)$ be the collection of maximal antichains of $\BB$ of size at most ${\omega_1}$. Then 
	\[
	\sg(\BB,M) = \bigwedge\bp{  \bigvee (A \cap M) :{A \in \A(\BB) \cap M}}
	\]
\end{proposition}
\begin{proof}
	Since $\A(\BB) \subseteq \PD(\BB)$, the inequality
	\[
	\sg(\BB,M) \leq \bigwedge\bp{ { \bigvee (A \cap M) }:{A \in \A(\BB) \cap M}}
	\]
	is trivial.
	Conversely, if $D=\bp{b_\alpha : \alpha <{\omega_1}} \in \PD(\BB)\cap M$, define
	\[
	A_D = \bp{a_\alpha = b_\alpha \wedge \neg \bigvee\bp{ b_\beta:{\beta < \alpha}} : \alpha < {\omega_1}}
	\]
	By elementarity, since $D \in M$ also $A_D$ is in $M$. It is straightforward to verify that $A_D$ is an antichain, and since $\bigvee A_D = \bigvee D = \1$ it is also maximal. Moreover, since $a_\alpha \leq b_\alpha$ we have that $\bigvee A_D \cap M \leq \bigvee D \cap M$. Thus, for any $D \in \PD(\BB)\cap M$, we have that $\bigwedge\bp { \bigvee (A \cap M) :{A \in \A(\BB) \cap M}} \leq \bigvee D \cap M$ hence
	\[
	\bigwedge\bp { \bigvee (A \cap M):{A \in \A(\BB) \cap M} } \leq \sg(\BB,M)
	\]
	The thesis follows.
\end{proof}

\begin{proposition}\label{4ristretto}
  Let $\BB$ be a complete boolean algebra and $M \prec H_\theta$ for some $\theta \gg \vp{\BB}$. Then for all $b \in M \cap \BB$
  \[
    \sg(\BB \res b, M )= \sg(\BB, M) \wedge b.
  \]
\end{proposition}

\begin{proof}
  Observe that if $A$ is a maximal antichain  in $\BB$, then $A \wedge b = \{a \wedge b : a \in A\}$ is a maximal antichain in $\BB \res b$. Moreover for each maximal antichain $A_b$ in $\BB\res b \cap M$, $A= A_b \cup \{\neg b\}$ is a maximal antichain in $\BB  \cap M$. Therefore
  \[
   \sg(\BB, M) \wedge b = \bigwedge \bigvee (A \cap M) \wedge b = \bigwedge \bigvee ((A \wedge b) \cap M) = \sg(\BB \res b, M).
  \]
\end{proof}

We are now ready to introduce the definition of semiproperness and properness
for complete boolean algebras and regular embeddings.

\begin{definition}\label{defSPboole}
	Let $\BB$ be a complete boolean algebra, $S$ be a stationary set on 
	$H_\theta$ with $\theta \gg \vp{\BB}$. $\BB$ is $S$-$\SP$ iff for club many 
	$M \in S$ whenever $b$ is in $\BB\cap M$, we have that $\sg(\BB,M) \wedge b > \0_\BB$.

		Similarly, $i:\BB \rightarrow \QQ$ is $S$-$\SP$ iff $\BB$ is $S$-$\SP$ and for 
		club many $M \in S$, whenever $c$ is in $\QQ \cap M$ we have that
	\[
	\pi(c \wedge \sg(\QQ,M))= \pi(c) \wedge \sg(\BB,M).
	\]
\end{definition}

The previous definitions can be reformulated with a well-known trick in the following form.

\begin{proposition} \label{4SPequivC}
	$\BB$ is $S$-$\SP$ iff for every $\nu \gg \theta$ regular, $M \prec H_{\nu}$ with $\BB, S \in M$ and $M \cap H_{\theta} \in S$ then $\forall b \in \BB \cap M$, $\sg(\BB,M) \wedge b > \0$.

	Similarly, $i:\BB \rightarrow \QQ$ is $S$-$\SP$ iff $\BB$ is $S$-$\SP$ and for every $\nu \gg \theta$ regular, $M \prec H_{\nu}$ with $i, S\in M$ and $M \cap H_{\theta} \in S$ then $\forall c \in \QQ \cap M$
	\[
	\pi(c \wedge \sg(\QQ,M))= \pi(c) \wedge \sg(\BB,M).
	\]
\end{proposition}
\begin{proof}
	First, suppose that $\BB$, $i:\BB \rightarrow \QQ$ satisfy the above conditions. Then $C=\{M \cap H_{\theta} : M \prec H_{\nu}, ~\BB,S\in M\}$ is a club (since it is the projection of a club), and witnesses that $\BB$, $i:\BB \rightarrow \QQ$ are $S$-$\SP$.

	Conversely, suppose that $\BB$, $i:\BB \rightarrow \QQ$ are $S$-$\SP$ and fix $\nu \gg \theta$ regular and $M \prec H_{\nu}$ with $\BB,S \in M$, $M \cap H_\theta \in S$. Since the sentence that $\BB$, $i:\BB \rightarrow \QQ$ are $S$-$\SP$ is entirely computable in $H_{\nu}$ and $M \prec H_{\nu}$, there exists a club $C \in M$ witnessing that $\BB$, $i:\BB \rightarrow \QQ$ are $S$-$\SP$.  Furthermore, $M$ models that $C$ is a club hence $M \cap H_\theta \in C$ and $\sg(\BB,M) \wedge b > \0$, $\pi(c \wedge \sg(\QQ,M))= \pi(c) \wedge \sg(\BB,M)$ hold for any $b \in \BB \cap M$, $c \in \QQ \cap M$ since $C$ witnesses that $\BB$, $i:\BB \rightarrow \QQ$ are $S$-$\SP$ and $M \cap H_\theta \in S \cap C$.
\end{proof}

We may observe that if $i: \BB \to \QQ$ is $S$-$\SP$, then $\QQ$ is $S$-$\SP$. 
As a matter of fact $c \in \QQ \cap M$ is such that $\sg(\QQ,M) \wedge c = \0$ iff
\[
	\0 =\pi(c \wedge \sg(\QQ,M))= \pi(c) \wedge \sg(\BB,M),
\]
this contradicts the assumption that $\BB$ is $S$-$\SP$.

\subsection{Shelah's semiproperness}

Definition~\ref{def:semipropcba} of semiproperness is equivalent to the usual Shelah's notion of semiproperness. In this section we spell out the details of this fact.

\begin{definition}(Shelah)
	Let $P$ be a partial order, and fix $M \prec H_\theta$. Then $q$ is a \emph{$M$-semigeneric} condition for $P$ iff for every $\dot{\alpha} \in V^{P}\cap M$ such that $\1_P \Vdash \dot{\alpha} < \check{\omega}_1$,
	\[
		q \Vdash \dot{\alpha} < M \cap \omega_1.
	\]

	$P$ is \emph{$S$-$\SP$ in the sense of Shelah} if there exists a club $C$ of elementary substructures of $H_\theta$ such that for every $M \in S \cap C$ there exists a $M$-semigeneric condition below every element of $P \cap M$.
\end{definition}

\begin{proposition} \label{4supSG}
	Let $\BB$ be a complete boolean algebra, and fix $M \prec H_\theta$. Then
	\[
	\sg(\BB,M) = \bigvee \bp{q \in \BB : ~ q \text{ is a }M\text{-semigeneric condition} }
	\]
\end{proposition}
\begin{proof}
	Given $A = \bp{a_\beta : ~ \beta < {\omega_1} } \in \A(\BB)$, define $\dot{\alpha}_A = \bp{ \ap{\check{\gamma},a_\beta} : ~ \gamma < \beta < {\omega_1} }$. It is straightforward to check that $\Qp{\dot{\alpha}_A < \check{\omega}_1} = \bigvee\bp{ \Qp{\dot{\alpha}_A = \check{\beta}}:{\beta < {\omega_1}}} = \bigvee\bp{ a_\beta:{\beta < {\omega_1}}} = \1$.
	Conversely, given $\dot{\alpha} \in V^\BB \cap M$ such that $\Qp{\dot{\alpha} < \check{\omega}_1} = \1$, define $A_{\dot{\alpha}} = \bp{a_\beta = \Qp{\dot{\alpha} = \check{\beta}} : ~ \beta < {\omega_1}}$. It is straightforward to check that $A_{\dot{\alpha}} \in \A(\BB)$.

	Suppose now that $q$ is a $M$-semigeneric condition, and fix an arbitrary $A \in \A(\BB) \cap M$. Then $\dot{\alpha}_A \in M$ and $\Qp{\dot{\alpha}_A < \check{\omega}_1} = \1$, hence
	\[\begin{split}
	&q \leq \Qp{\dot{\alpha}_A < \check{\cp{M \cap {\omega_1}}}} = \bigvee\bp{ \Qp{\dot{\alpha}_A = \check{\beta}}:{\beta \in M \cap {\omega_1}}} \\
	&= \bigvee\bp{a_\beta:{\beta \in M \cap {\omega_1}} } = \bigvee A \cap M
	\end{split}
	\]
	It follows that $q \leq \bigwedge\bp{ \bigvee \cp{A \cap M}:{A \in \A(\BB) \cap M}} = \sg(\BB,M)$, hence
	\[
	\sg(\BB,M) \geq \bigvee \bp{q \in \BB : ~ q \text{ is a }M\text{-semigeneric condition} }
	\]

	Finally, we show that $\sg(\BB,M)$ is a $M$-semigeneric condition itself. Fix an arbitrary $\dot{\alpha} \in V^\BB \cap M$ such that $\1_\BB \Vdash \dot{\alpha} < \check{\omega}_1$, and let $A_{\dot{\alpha}} \in A_{\omega_1}(\BB)$ be as above. Since $\dot{\alpha} \in M$, also $A_{\dot{\alpha}} \in M$. Moreover,
	\[\begin{split}
	&\Qp{\dot{\alpha} < \check{\cp{M \cap {\omega_1}}}} = \bigvee\bp{ \Qp{\dot{\alpha} = \check{\beta}}:{\beta \in M \cap {\omega_1}}} \\
	&= \bigvee\bp{ a_\beta:{\beta \in M \cap {\omega_1}}} = \bigvee A_{\dot{\alpha}} \cap M \geq \sg(\BB,M)
	\end{split}
	\]
	concluding the proof.
\end{proof}

\begin{corollary}\label{equivSP}
	Let $P$ be a partial order, then $P$ is $S$-$\SP$ in the sense of Shelah if and only if $\RO(P)$ is $S$-$\SP$.
\end{corollary}
\begin{proof}
	First, suppose that $P$ is $S$-$\SP$ in the sense of Shelah as witnessed by $C$, and fix $M \in S \cap C$, $b \in \RO(P) \cap M$. Since $P$ is dense in $\RO(P)$, there exists a $p \in P \cap M$, $p \leq b$, and by semiproperness there exists a $q \in P$, $q \leq p \leq b$ that is $M$-semigeneric. Then $q > \0$ and by Proposition \ref{4supSG}, $q \leq \sg(\RO(P),M)$. Hence $\sg(\RO(P),M) \wedge b \geq q > \0$.

	Finally, suppose that $\RO(P)$ is $S$-$\SP$ as witnessed by $C$, and fix $M \in S \cap C$, $p \in P \cap M$. Since $P$ is dense in $\RO(P)$, there exists a $q \in P$, $q \leq \sg(\RO(P),M) \wedge p$ that is a $M$-semigeneric condition since $q \leq \sg(\RO(P),M)$ and the set of semigeneric conditions is open.
\end{proof}


\subsection{Topological characterization of semiproperness} \label{rap2}

An equivalent definition of semiproperness and properness can be stated also in the topological context introduced in the previous sections, as a Baire Category property. 
Let $\BB$ be a complete boolean algebra and $X_{\BB}$ be the space of its ultrafilters defined 
in~\ref{def:ROP}.
The Baire Category Theorem states that given any family of maximal antichains $\{A_n: n \in \omega \}$
of $\BB$, then $\bigcap \bigcup\{N_a : a \in A_n\}$ is comeager in $X_{\BB}$, so 
\[
 \mathring{\overline{\bigcap_{n \in \omega} \bigcup\{N_a : a \in A_n\}}}=X_{\mathbb{B}}.
\]

Now, let $M \prec H_{\theta}$, $\mathbb{B}\in M$, then if $\{A_n: n \in \omega\}$ is a subset of the set of the maximal antichains of $\mathbb{B} \in M$, the classical construction of an $M$-generic filter shows that
\[
 \bigcap_{n \in \omega} \cp{\bigcup \{N_a : a \in A_n\cap M\} } \neq \emptyset.
\]
However this does not guarantee that
\[
 \bigcap_{n \in \omega} \cp{\bigcup \{N_a : a \in A_n\cap M\}} \mbox{ is comeager on some } N_b \mbox{ in } V.
\]
This latter requirement is exactly the request that $\BB$ is proper:
\begin{definition}\label{defSPboole}
	Let $\BB$ be a complete boolean algebra, $M\prec H_\theta$ with $\theta \gg \vp{\BB}$.
	\[
	\gen(\BB,M) = \bigwedge\bp{\bigvee (A \cap M):A\in M\text{ is a maximal antichain of }\BB }
	\]
	is the \emph{degree of genericity} of $M$ with respect to $\BB$. 
	
	$\BB$ is \emph{proper} iff for club many countable models $M\prec H_\theta$, 
	whenever $b$ is in $\BB\cap M$, we have that $\gen(\BB,M) \wedge b > \0_\BB$.
\end{definition}
We leave to the reader to check (along the same lines of what has been done for semiproperness) that this algebraic definition of properness is equivalent to the usual one by Shelah.

\begin{proposition}
$\mathbb{B}$ is proper if and only if $\forall M \prec H_{\theta}$ with $\mathbb{B} \in M$, $M$ countable
\[X_M= \bigcap \bp{\bigcup \{N_a : a \in A \cap M\}: A \in M \mbox{ maximal antichain of } \mathbb{B} }\]
is such that $\forall c \in M \cap \mathbb{B} \exists b \in \mathbb{B}$ such that $X_M$ is comeager set on $N_b \cap N_c$.
\end{proposition} 
\begin{proof}
As a matter of fact
\[ \forall c \in M \cap \mathbb{B} \ \exists b (N_b \subseteq \mathring{\overline{ X_M}} \cap N_c)\]
\[\iff \forall c \in M \cap \mathbb{B} \exists b \leq \bigwedge \bp{\bigvee (A \cap M) : A \in M \mbox{ maximal antichain }} \wedge c. \]
\end{proof}

\begin{proposition}
$\mathbb{B}$ is semiproper if and only if $\forall M \prec H_{\theta}$ with $\mathbb{B} \in M$, $M$ countable
\[X_M= \bigcap \bp{\bigcup \{N_a : a \in A \cap M\}: A \in M \mbox{ maximal antichain of } \mathbb{B}, |A|=\omega_1}\]
is such that $\forall c \in M \cap \mathbb{B} \exists b \in \mathbb{B}$ such that $X_M$ is comeager set on $N_b \wedge N_c$.
\end{proposition} 

%% file: 7_SPIterations.tex
\section{Semiproper iterations}\label{sec:semiproperiter}

In this section we will prove that (granting some natural assumptions) 
the iteration of semiproper boolean algebras is semiproper.
First we shall examine the case of two-step iterations, then we will focus on the limit case.

\subsection{Two-step iterations}

The notion of being $S$-$\SP$ can change when we move to a generic extension: for example, $S$ can be no longer stationary. In order to recover the ``stationarity'' in 
$V[G]$ of an $S$ which is stationary in $V$,
we are led to the following definition:

\begin{definition}
	Let $S$ be a subset of $\pp(H_\theta)$, $\BB\in H_{\theta}$ be a complete boolean algebra, and $G$ a $V$-generic filter for $\BB$. We define
	\[
 		S(G)= \{M[G] : \BB \in M \in S\}.
	\]
\end{definition}

\begin{fact}
Let $S$ be a stationary set on $H_\theta$, $\BB\in H_{\theta}$ be a complete boolean algebra, and $G$ be a $V$-generic filter for $\BB$. Then $S(G)$ is stationary in $V[G]$.
\end{fact}

\begin{proof}
	Let $\dot{C} \in V^\BB$ be a name for a club on $P(H_\theta)$, and let $M \prec H_{\theta^+}$ be such that $M \cap H_\theta \in S$, $\BB,\dot{C} \in M$. Then $C \in M[G]$ hence $M[G] \cap H_\theta \in C$, and $M[G] \cap H_\theta = (M \cap H_\theta)[G]$ thus $M[G] \cap H_\theta \in S(G) \cap C$.
\end{proof}

\begin{proposition}\label{6iSP}
	Let $\BB$ be a $S$-$\SP$ complete boolean algebra, and let $\dot{\QQ}$ be such that
	\[ 
  		\Qp{ \dot{\QQ} \mbox{ is } S(\dot{G})\mbox{-}\SP } =\1,
	\]
	then $\DD = \BB \ast \dot{\QQ}$ and $i_{\BB \ast \dot{\QQ}}$ are $S$-$\SP$.
\end{proposition}

\begin{proof}
	First, we verify that $i = i_{\BB \ast \dot{\QQ}}$ is $S$-$\SP$. Let $\dot{C}_1$ be the club that witnesses $\Qp{\dot{\QQ} \mbox{ is } S(\dot{G})\mbox{-}\SP}=\1$, and let $M$ be such that $\dot{C}_1 \in M$: this guarantees that $V[G] \vDash M[{G}] \cap H_{\theta}^{V[{G}]} \in {C}_{1_{G}}$.

	 We shall first prove that $\pi(\sg(\DD,M))=\sg(\BB, M)$.
	Thanks to Lemma \ref{eRetrDense} we obtain $\sg(\DD,M) \leq i(\sg(\BB, M))$, hence
	\[
	 	\pi(\sg(\DD, M)) \leq (\sg(\BB, M)).
	\]

	Now we have to prove $\pi(\sg(\DD, M)) \geq (\sg(\BB, M))$. Let $\sg(\BB, M) \in G$, then, thanks to the semiproperness of $\BB$, $V[G] \vDash M \cap \omega_1 = M[{G}] \cap \omega_1.$ Therefore, thanks to Lemma \ref{6antichainDQ}, $V[G] \vDash [\sg(\DD, M)]_{i[{G}]} = \sg({\QQ}, M[{G}])$, hence
	\[
		 \Qp{ [\sg(\DD, M)]_{i[\dot{G}]} = \sg(\dot{\QQ}, M[\dot{G}]) } \geq \sg(\BB, M).
	\]
	This implies that if $\sg(\BB, M) \in G$,
	\[
		 \sg(\BB,M) \wedge \Qp{ [sg(\DD, M)]_{i[\dot{G}]} > \dot{0} } = \sg(\BB,M) \wedge \Qp{ \sg(\dot{\QQ}, M[\dot{G}]) > \dot{0} } = \sg(\BB,M),
	\]
	using the semiproperness of $\dot{\QQ}$ in $V[G]$. Thus,
	\[
 		\pi(\sg(\DD, M)) = \Qp{[\sg(\DD, M)]_{i[\dot{G}]} > \dot{0} } \geq \sg(\BB, M).
	\]
	Finally, by Lemma \ref{4ristretto} and \ref{6StepRestr}, repeating the proof for $\BB\res \pi([\dot{c}])$ and $\DD \res [\dot{c}]$ (that are a two-step iteration of $S$-$\SP$ boolean algebras) we obtain that
	\[
		\pi(\sg(\DD,M) \wedge [\dot{c}])= \sg(\BB, M) \wedge \pi([\dot{c}])
	\]
	hence $i$ is $S$-$\SP$. Moreover, for any $[\dot{c}] \in \DD \cap M$ incompatible with $\sg(\DD,M)$,
	\[
		\pi(\0) = \pi(\sg(\DD,M )\wedge [\dot{c}]) = \sg(\BB, M)\wedge \pi([\dot{c}]) = \0
	\]
	that implies $\pi([\dot{c}]) = \0$ and $[\dot{c}] = \0$ since $\BB$ is $S$-$\SP$, completing the proof that $\DD$ is $S$-$\SP$.
\end{proof}

%

\begin{lemma}\label{EmbTwoStepSP}
	Let $\BB$, $\QQ_0$, $\QQ_1$ be $S$-$\SP$ complete boolean algebras, and let $G$ be any $V$-generic filter for $\BB$. Let $i_0$, $i_1$, $j$ form a commutative diagram of regular embeddings as in the following picture:
	\[
	\begin{tikzpicture}[xscale=1.5,yscale=-1.2]
		\node (B) at (0, 0) {$\BB$};
		\node (Q0) at (1, 0) {$\QQ_0$};
		\node (Q1) at (1, 1) {$\QQ_1$};
		\path (B) edge [->]node [auto] {$\scriptstyle{i_0}$} (Q0);
		\path (B) edge [->]node [auto,swap] {$\scriptstyle{i_1}$} (Q1);
		\path (Q0) edge [->]node [auto] {$\scriptstyle{j}$} (Q1);
	\end{tikzpicture}
	\]
	Moreover assume that $\QQ_0/i_0[G]$ is $S(G)$-$\SP$ and
	\[
	\Qp{\QQ_1/j[\dot{G}_{\QQ_0}]\mbox{ is }S(\dot{G}_{\QQ_0})\mbox{-}\SP}_{\QQ_0}=\1_{\QQ_0}.
	\]
	Then in $V[G]$, $j/_G: \QQ_0/_G \to \QQ_1/_G$ is an $S(G)$-$\SP$ embedding.
\end{lemma}

\begin{proof}
	Let $G$ be $V$-generic for $\BB$. Pick $K$ $V[G]$-generic for $\QQ_0/_G$. Then we can let 
	\[
		H=\{c \in \QQ_0 : [c]_G \in K\},
	\]
	and we get that 
	\[
		K=H/_G= \{[c]_G: c\in H\}.
	\]
	Moreover $H$ is $V$-generic for $\QQ_0$, 
	$V[H]=V[G][H/_G]$ and in $V[H]$ we have that $S(H)=S(G)(H/_G)$. 
	Since this latter equality holds for whichever choice of $K$ we make, 
	this gives that in $V[G]$ it holds that $j/_G: \QQ_0/_G \to \QQ_1/_G$ is a map such that
	\[
		\Qp{(\QQ_1/_G)/_{j/_G[\dot{G}_{\QQ_0/_G}]} \mbox{ is } S(G)(\dot{G}_{\QQ_0/_G})\mbox{-}\SP}_{\QQ_0/_G}=\1_{\QQ_0/_G}.
	\]
	So, by applying Proposition \ref{6iSP}, $j/G$ is $S(G)$-$\SP$ in $V[G]$.
\end{proof}

\subsection{Semiproper iteration systems}

The limit case needs a slightly different approach depending on the length of the iteration. We shall start with some general lemmas, then we will proceed to examine the different cases one by one.

\begin{definition}
	An iteration system $\FFF = \{i_{\alpha\beta} : \alpha \leq \beta < \lambda\}$ is $S$-$\SP$ iff $i_{\alpha\beta}$ is $S$-$\SP$ for all $\alpha \leq \beta < \lambda$.

	An iteration system $\FFF = \{i_{\alpha\beta} : \alpha \leq \beta < \lambda\}$ is $\RCS$ iff for all $\alpha < \lambda$ limit ordinal we have $\BB_\alpha = \RO(\RCS(\FFF \res \alpha))$.
\end{definition}

\begin{fact}\label{6IterRestr}
	Let $\FFF = \{i_{\alpha\beta} : \alpha \leq \beta < \lambda\}$ be an $S$-$\SP$ iteration system, $f$ be in $T(\FFF)$. Then 
	\[
		\FFF \res f = \{(i_{\alpha\beta})_{f(\beta)}:\BB_{\alpha}\res f(\alpha)\to \BB_\beta\res f(\beta) : \alpha \leq \beta < \lambda\}
	\]
	 is an $S$-$\SP$ iteration system and its associated retractions are the restriction of the original retractions.
\end{fact}

\begin{lemma}\label{LemmaGenerale}
	Let $\FFF=\bp{i_{\alpha\beta}: \BB_\alpha \to \BB_\beta: \alpha \leq \beta < \lambda}$ be an $\RCS$ and $S$-$\SP$ iteration system with $S$ stationary on $[H_\theta]^\omega$. Let $M$ be in $S$, $g \in M$ be any condition in $\RCS(\FFF)$, $\dot{\alpha} \in M$ be a name for a countable ordinal, $\delta \in M$ be an ordinal smaller than $\lambda$.

	Then there exists a condition $g' \in \RCS(\FFF) \cap M$ below $g$ with $g'(\delta) = g(\delta)$ and $g' \wedge i_\delta(\sg(\BB_\delta,M))$ forces that $\dot{\alpha} < M \cap \omega_1$.
	If $\lambda = \omega_1$,  then the support of $g' \wedge i_\delta(\sg(\BB_\delta,M))$ is contained in $M \cap \omega_1$.
\end{lemma}
\begin{proof}
Let $D \in M$ be the set of conditions in $\RCS(\FFF)$ deciding the value of $\dot{\alpha}$ ($D$ is open dense by the forcing theorem):
	\[
	D = \{f \in \RCS(\FFF) : \exists \beta < \omega_1 \ f \Vdash \dot{\alpha}=\check{\beta}\}.
	\]

	Consider the set $\pi_\delta[D \res g]$ (which is open dense below $g(\delta)$ by Lemma \ref{eRetrDense}) and fix $A$ a maximal antichain in $M$ contained in it, so that $\bigvee A = g(\delta)$. Let $\phi: A \to D \res g$ be a map in $M$ such that $\pi_\delta(\phi(a))=a$ for every $a\in A$, and define $g' \in \RCS(\FFF) \cap M$ by $g' =  \tilde{\bigvee} \phi[A]$. Observe that $g'(\delta) = g(\delta)$ by definition of pointwise supremum and $g' \leq g$ since $\tilde{\bigvee} \phi[A]$ is really the supremum of $\phi[A]$ in $\RO(T(\FFF))$ by Lemma \ref{iTFSup} (thus it is the supremum in $\RO(\RCS(\FFF))$ as well).

	Then we can define a name\footnote{Literally speaking this is not a $\BB_\delta$-name according to our definition. See the footnote below~\ref{def:forcingnames} to resolve this ambiguity.} $\dot{\beta} \in V^{\BB_\delta} \cap M$ as:
	\[
	\dot{\beta} = \bp{ \ap{\check{\gamma}, a}: ~ a \in A, ~ \phi(a) \Vdash_{\RCS(\FFF)} \dot{\alpha} > \check{\gamma} }
	\]
	so that for any $a \in A$, $a \Vdash \dot{\beta} = \check{\xi}$ iff $\phi(a) \Vdash \dot{\alpha} = \check{\xi}$. It follows that $\Qp{\hat{\imath}_\delta(\dot{\beta}) = \dot{\alpha}} \geq \bigvee \phi[A] = g'$. Moreover, $\sg(\BB_\delta,M) \leq \Qp{\dot{\beta} < \check{M \cap \omega_1}}$ and is compatible with $g'(\delta) \in M$ (since $\BB_\delta$ is $S$-$\SP$), so that
	\[
	\Qp{\dot{\alpha} < \check{M \cap \omega_1} } \geq g' \wedge i_{\delta}(\sg(\BB_\delta,M)).
	\]

	If $\lambda = \omega_1$, $\RCS(\FFF) = C(\FFF)$ and we can define a name $\dot{\gamma} \in V^{\BB_\delta} \cap M$ for a countable ordinal setting:
	\[
	\dot{\gamma} = \bp{ \ap{\check{\eta}, a}: ~ a \in A, ~ \eta < \supp(\phi(a)) }.
	\]
	Notice that 
	$\dot{\gamma}$ is defined in such a way that for all $\beta<\omega_1$ 
	\[
	\Qp{\dot{\gamma}=\beta}=\bigvee\{a\in A:\supp(\phi(a))=\beta\}.
	\]
	 In particular this gives that:
	 \begin{align*}
	 i_\delta(\Qp{\dot{\gamma}<\beta})\wedge g'=\\
	 = i_\delta(\bigvee\{a\in A:\supp(\phi(a))<\beta\})\wedge\bigvee\{\phi(a):a\in A\}=\\
	 =\tilde{\bigvee}\{\phi(a):a\in A,\,\supp(\phi(a))<\beta\}.
	 \end{align*}
	 
	 Now observe that 
	 \[
	 g'\wedge \sg(\BB_{\delta},M)=\tilde{\bigvee}\{\phi(a)\wedge i_\delta(\sg(\BB_{\delta},M)):a\in A\}.
	 \]
	 Since $\sg(\BB_{\delta},M) \leq \Qp{\dot{\gamma} < \check{M \cap \omega_1} }$, we get that:
	 \begin{align*}
	 g'\wedge i_\delta(\sg(\BB_{\delta},M))=\\
	 =g'\wedge i_\delta(\sg(\BB_{\delta},M))\wedge i_\delta(\Qp{\dot{\gamma}<M\cap\omega_1})=\\
	 =\tilde{\bigvee}\{\phi(a):a\in A,\,\supp(\phi(a))<M\cap\omega_1\}\wedge i_\delta(\sg(\BB_{\delta},M)).
	 \end{align*}
	 It is now immediate to check that this latter element of $C(\FFF)$ has support contained in
	 $M\cap\omega_1$ as required.
\end{proof}

\begin{lemma}\label{LemmaMinOmega1}
	Let $\FFF = \{i_{nm} : n \leq m < \omega\}$ be an $S$-$\SP$ iteration system with $S$ stationary on $[H_\theta]^\omega$. Then $T(\FFF)$ and the corresponding $i_{n\omega}$ are $S$-$\SP$.
\end{lemma}
\begin{proof}
	By Proposition \ref{4SPequivC}, any countable $M \prec H_\nu$ with $\nu > \theta$, $\FFF,S \in M$, $M \cap H_\theta \in S$, witnesses the semiproperness of every $i_{nm}$.

	We need to show that for every $f \in T(\FFF) \cap M$, $n < \omega$,
	\[
	\pi_{n\omega}(\sg(\RO(T(\FFF)),M) \wedge f) = \sg(\BB_n,M) \wedge f(n)
	\]
	this would also imply that $\RO(T(\FFF))$is $S$-$\SP$ by the same reasoning of the proof of Lemma \ref{6iSP}. Without loss of generality, we can assume that $n=0$ and by Lemma \ref{4ristretto} and \ref{6IterRestr} we can also assume that $f = \1$. Thus is sufficient to prove that
	\[
	\pi_{0\omega}(\sg(\RO(T(\FFF)),M)) = \sg(\BB_0,M)
	\]

	Let $\{\dot{\alpha}_n : n \in \omega\}$ be an enumeration of the $T(\FFF)$-names in $M$ for countable ordinals. Let $g_0 = \1_{T(\FFF)}$, $g_{n+1}$ be obtained from $g_n$, $\dot{\alpha}_n$, $n$ as in Lemma \ref{LemmaGenerale}, so that
	\[
	\Qp{\dot{\alpha}_n < \check{M \cap \omega_1} } \geq g_{n+1} \wedge i_n(\sg(\BB_n,M))
	\]

	Consider now the sequence $\bar{g}(n) = g_n(n) \wedge \sg(\BB_n,M)$. This sequence is a thread since $i_{n,n+1}$ is $S$-$\SP$ and $g_n(n) \in M$ for every $n$, hence
	\[
	\pi_{n,n+1}( \sg(\BB_{n+1},M) \wedge g_{n+1}(n+1) ) = \sg(\BB_n,M) \wedge \pi_{n,n+1}(g_{n+1}(n+1))
	\]
	and $\pi_{n,n+1}(g_{n+1}(n+1)) = g_{n+1}(n) = g_n(n) $ by Lemma \ref{LemmaGenerale}. Furthermore, for every $n\in \omega$, $\bar{g} \leq g_n$ since the sequence $g_n$ is decreasing, and $\bar{g} \leq i_{n}(\sg(\BB_n,M))$ since $\bar{g}(n) \leq \sg(\BB_n,M)$. It follows that $\bar{g}$ forces that $\Qp{\dot{\alpha}_n < \check{M \cap \omega_1}}$ for every $n$, thus $\bar{g} \leq \sg(\RO(T(\FFF)),M)$ by Lemma \ref{4supSG}. Then,
	\[
	\pi_{0}(\sg(\RO(T(\FFF)),M)) \geq \bar{g}(0) = g_0(0) \wedge \sg(\BB_0,M) = \sg(\BB_0,M)
	\]
	and the opposite inequality is trivial, completing the proof.
\end{proof}

\begin{lemma}\label{LemmaOmega1}
	Let $\FFF=\bp{i_{\alpha\beta}: \BB_\alpha \to \BB_\beta: \alpha \leq \beta < \omega_1}$ be an $\RCS$ and $S$-$\SP$ iteration system with $S$ stationary on $[H_\theta]^\omega$. Then $C(\FFF)$ and the corresponding $i_{\alpha\omega_1}$ are $S$-$\SP$.
\end{lemma}
\begin{proof}
	The proof follows the same pattern of the previous Lemma \ref{LemmaMinOmega1}. By Proposition \ref{4SPequivC}, any countable $M \prec H_\nu$ with $\nu > \theta$, $\FFF,S \in M$, $M \cap H_\theta \in S$, witnesses the semiproperness of every $i_{\alpha\beta}$ with $\alpha,\beta \in M \cap \omega_1$.

	As before, by Lemma \ref{4ristretto} and \ref{6IterRestr} we only need to show that
	\[
	\pi_0(\sg(\RO(C(\FFF)),M)) \geq \sg(\BB_0,M),
	\]
	the other inequality being trivial.
	Let $\ap{\delta_n: n \in \omega}$ be an increasing sequence of ordinals such that $\delta_0 = 0$ and $\sup_n \delta_n = \delta = M \cap \omega_1$, and $\{\dot{\alpha}_n : n \in \omega\}$ be an enumeration of the $C(\FFF)$-names in $M$ for countable ordinals. Let $g_0 = \1_{T(\FFF)}$, $g_{n+1}$ be obtained from $g_n$, $\dot{\alpha}_n$, $\delta_n$ as in Lemma \ref{LemmaGenerale}, so that
	\[
	\Qp{\dot{\alpha}_n < \check{M \cap \omega_1} } \geq g_{n+1} \wedge i_{\delta_n}(\sg(\BB_{\delta_n},M)).
	\]
	Consider now the sequence $\bar{g}(\delta_n) = g_n(\delta_n) \wedge \sg(\BB_{\delta_n},M)$. As before, this sequence induces a thread on $\FFF \res \delta$, so that $\bar{g} \in \BB_\delta$ since $\FFF$ is an $\RCS$-iteration, $\delta$ has countable cofinality and thus we can naturally identify 
	$T(\FFF\res\delta)$ as a dense subset of $\BB_\delta$.
	Moreover we can also check that $i_\delta(\bar{g})$ is a thread in $C(\FFF)$ with support $\delta$
	such that $i_\delta(\bar{g})(\alpha)=\bar{g}(\alpha)$
	for all $\alpha<\delta$.
	
	Since by Lemma \ref{LemmaGenerale}
	\[
	\supp(g_{n+1} \wedge i_{\delta_n}(\sg(\BB_{\delta_n},M))) \leq \delta,
	\]
	the relation $i_\delta(\bar{g}) \leq g_{n+1} \wedge i_{\delta_n}(\sg(\BB_{\delta_n},M))$ holds pointwise hence $i_\delta(\bar{g})$ forces that $\Qp{\dot{\alpha}_n < \check{M \cap \omega_1}}$ for every $n$. Thus, $i_\delta(\bar{g}) \leq \sg(\RO(C(\FFF)),M)$ by Lemma \ref{4supSG} and $\pi_{0}(\sg(\RO(T(\FFF)),M)) \geq \bar{g}(0) = g_0(0) \wedge \sg(\BB_0,M) = \sg(\BB_0,M)$ as required.
\end{proof}

\begin{lemma} \label{LemmaLCC}
	Let $\FFF=\bp{i_{\alpha\beta}: \BB_\alpha \to \BB_\beta: \alpha \leq \beta < \lambda}$ be an $\RCS$ and $S$-$\SP$ iteration system with $S$ stationary on $[H_\theta]^\omega$ such that $C(\FFF)$ is ${<}\lambda$-cc. Then $C(\FFF)$ and the corresponding $i_{\alpha\lambda}$ are $S$-$\SP$.
\end{lemma}

\begin{proof}
	The proof follows the same pattern of the previous Lemmas \ref{LemmaMinOmega1} and \ref{LemmaOmega1}. By Proposition \ref{4SPequivC}, any countable $M \prec H_\nu$ with $\nu > \theta$, $\FFF,S \in M$, $M \cap H_\theta \in S$, witnesses the semiproperness of every $i_{\alpha\beta}$ with $\alpha,\beta \in M \cap \lambda$.

	As before, by Lemma \ref{4ristretto} and \ref{6IterRestr} we only need to show that
	\[
	\pi_0(\sg(\RO(C(\FFF)),M)) \geq \sg(\BB_0,M).
	\]
	Let $\ap{\delta_n: n \in \omega}$ be an increasing sequence of ordinals such that $\delta_0 = 0$ and $\sup_n \delta_n = \delta = \sup(M \cap \lambda)$, and $\{\dot{\alpha}_n : n \in \omega\}$ be an enumeration of the $C(\FFF)$-names in $M$ for countable ordinals. Let $g_0 = \1_{T(\FFF)}$, $g_{n+1}$ be obtained from $g_n$, $\dot{\alpha}_n$, $\delta_n$ as in Lemma \ref{LemmaGenerale}, so that
	\[
	\Qp{\dot{\alpha}_n < \check{M \cap \omega_1} } \geq g_{n+1} \wedge i_{\delta_n}(\sg(\BB_{\delta_n},M)).
	\]
	Since $C(\FFF)$ is ${<}\lambda$-cc by Theorem \ref{iBaumCorollary} we have that $T(\FFF) = \RO(C(\FFF)) = C(\FFF)$, so every $g_n$ is in $C(\FFF) \cap M$ hence $M$ has to model $g_n$ to be eventually constant, thus $\supp(g_n) < \delta$. Then the sequence $\bar{g}(\delta_n) = g_n(\delta_n) \wedge \sg(\BB_{\delta_n},M)$ induces a thread on $\FFF \res \delta$ 
	(hence $\bar{g} \in \BB_\delta=\RO(T(\FFF\res\delta))$ by the countable cofinality of $\delta$) 
	and $i_\delta(\bar{g}) \leq g_{n+1} \wedge i_{\delta_n}(\sg(\BB_{\delta_n},M))$ for every $n$, so that $i_\delta(\bar{g}) \leq \sg(\RO(C(\FFF)),M)$ by Lemma \ref{4supSG} and $\pi_{0}(\sg(\RO(T(\FFF)),M)) \geq \bar{g}(0) = g_0(0) \wedge \sg(\BB_0,M) = \sg(\BB_0,M)$ as required.
\end{proof}

\begin{theorem} \label{ThmIteration}
	Let $\FFF=\bp{i_{\alpha\beta}: \BB_\alpha \to \BB_\beta: \alpha \leq \beta < \lambda}$ be an $\RCS$ and $S$-$\SP$ iteration system with $S$ stationary on $[H_\theta]^\omega$, such that for all $\alpha < \beta < \lambda$,
	\[
	\Qp{ \BB_\beta/i_{\alpha\beta}[\dot{G}_\alpha] \text{ is } S(\dot{G}_\alpha)\text{-}\SP } = \1_{\BB_\alpha}
	\]
	and for all $\alpha$ there is a $\beta > \alpha$ such that $\BB_\beta \Vdash \vp{\BB_\alpha} \leq \omega_1$. Then $\RCS(\FFF)$ and the corresponding $i_{\alpha\lambda}$ are $S$-$\SP$.
\end{theorem}

\begin{proof}
	First, suppose that for all $\alpha$ we have that $\vp{\BB_\alpha} < \lambda$. Then, by Theorem \ref{iBaumgartner}, $C(\FFF)$ is ${<}\lambda$-cc and $\RCS(\FFF) = C(\FFF)$ hence by Lemma \ref{LemmaLCC} we have the thesis.

	Now suppose that there is an $\alpha$ such that $\vp{\BB_\alpha} \geq \lambda$. Then by hypothesis there is a $\beta > \alpha$ such that $\BB_\beta \Vdash \vp{\BB_\alpha} \leq \omega_1$, thus $\BB_\beta \Vdash \cf{\lambda} \leq \omega_1$. So by Lemma \ref{EmbTwoStepSP} $\FFF /_{\dot{G}_\beta}$  is a $\BB_{\beta}$-name for an $S(\dot{G}_\beta)$-$\SP$ iteration system that is equivalent to a system of length $\omega$ or $\omega_1$ hence its limit is $S(\dot{G}_\beta)$-$\SP$ by Lemma \ref{LemmaMinOmega1} or Lemma \ref{LemmaOmega1} applied in $V^{\BB_\beta}$. 
	Finally, $\RCS(\FFF)$ can always be factored as a two-step iteration of $\BB_\beta$ and $\RCS(\FFF /_{\dot{G}_\beta})$, hence by Proposition \ref{6iSP} we have the thesis.
\end{proof}

%% file: 8_MM.tex
\section{Consistency of $\MM$}\label{sec:consMM}

In this section we will see one of the main applications of the general results about semiproperness and iterations, namely that assuming a supercompact cardinal it is possible to force the forcing axiom
$\MM$ (Martin's maximum).

\begin{definition}
	A cardinal $\delta$ is \emph{supercompact} and $f: ~ \delta \rightarrow V_\delta$ is its \emph{Laver function} iff for every set $X$ there exists an elementary embedding $j: ~ V_\alpha \rightarrow V_\lambda$ such that $j(f(\crit(j))) = X$, $j(crit(j)) = \delta$.
\end{definition}

\begin{definition}
	$\FA_\kappa(\PP)$ holds if for every $\mathcal D \subset \pp(\PP)$ family of open dense sets of $\PP$ with $\vp{\mathcal D} \leq \kappa$, there exists a filter $G\subset \PP$ such that $G \cap D \neq \emptyset$ for all $D \in \mathcal{D}$.
\end{definition}

\begin{definition}
	$\SPFA$ (semiproper forcing axiom) states that $\FA_{\omega_1}(\PP)$ holds for every 
	semiproper $\PP$.
\end{definition}

\begin{remark}
	It is worth noting that $\SPFA$ is in fact equivalent to $\MM$ (i.e. the sentence ``$\FA_{\omega_1}(\PP)$ hold for every $\PP$ stationary set preserving'').
\end{remark}

\begin{theorem}[Magidor, Foreman, Shelah]
	If $\delta$ is supercompact then there exists an $\RCS$ iteration $\FFF=\{i_{\alpha, \beta}: \BB_\alpha \rightarrow \BB_\beta: \alpha \leq \beta < \delta \}$ such that $\RCS(\FFF) \Vdash \SPFA$, collapses $\delta$ to $\omega_2$ and is ${<}\delta$-cc.
\end{theorem}
\begin{proof}
	Let $f: ~ \delta \rightarrow V_\delta$ be a Laver function for $\delta$. Define $\FFF=\{i_{\alpha, \beta}: \BB_\alpha \rightarrow \BB_\beta: \alpha \leq \beta < \delta\}$ such that $\BB_0 = \2$, $\BB_{2\cdot\alpha+1} = \BB_{2\cdot\alpha} \ast f(\alpha)$ if $f(\alpha)$ is a $\BB_{2\cdot\alpha}$-name for a semiproper poset and $\BB_{2\cdot\alpha}$ otherwise, $\BB_{2\cdot\alpha+2} = \BB_{2\cdot\alpha +1} \ast \Coll(\vp{\BB_\alpha},\omega_1)$ (as calculated in $\BB_{2\cdot\alpha +1}$).

	Then $\BB_\delta = \RCS(\FFF)$ is semiproper by Lemma \ref{ThmIteration}, and $\RCS(\FFF) = C(\FFF)$ is ${<}\delta$-cc by Theorem \ref{iBaumgartner} since $f$ has values in $V_\delta$ hence all $\BB_\alpha$ are in $V_\delta$. Furthermore,  $\BB_\delta$ preserves the regularity of $\delta$ since it is 
	${<}\delta$-cc and collapses $\alpha < \delta$ to $\omega_1$ (as guaranteed at steps $\BB_{2\cdot\alpha+2}$). Thus, we only need to prove that $\BB_\delta$ forces $\FA_{\omega_1} (\dot{Q})$ for any $\dot{Q}$ name for a semiproper poset.

	Fix a name $\dot{Q}$ for a semiproper poset, and find $j: V_\alpha \rightarrow V_\lambda$ such that $\crit(j) = \gamma$, $j(\gamma) = \delta$, $j(f(\gamma)) = \dot{Q}$. Notice that 
	$j(\BB_\gamma) = \BB_\delta$ (since $\BB_\gamma$ is defined by $f\restriction\gamma$ as $\BB_\delta$ is defined by $f$), $j( f(\gamma) ) = \dot{Q}$ is a $\BB_\delta$-name for a semiproper poset, hence by elementarity $\BB_\gamma=C(\FFF\restriction\gamma)$ is ${<}\gamma$-cc and
	 $f(\gamma)$ is a $\BB_\gamma$-name for a semiproper poset so $\BB_{\gamma+1} = \BB_\gamma \ast f(\gamma)$ (since $\gamma$ is limit hence $2\cdot\gamma = \gamma$).

	Let $G$ be $V$-generic for $\BB_\delta$, so that 
	$G_\gamma = \pi_\gamma[G] = \bp{ h(\gamma): ~ h \in G }$ is $V$-generic for $\BB_\gamma$ and $H = \bp{\val_{G_\gamma}(h(\gamma+1)): ~ h \in G }$ is $V[G_\gamma]$-generic for $\val_{G_\gamma}(f(\gamma))$. 
	Now observe that $G_\gamma=\{h\restriction\gamma:h\in G\}$ since 
	$\BB_\gamma=C(\FFF\restriction\gamma)$.
	
	In particular 
	\[
	j[G_\gamma]=j[\{h\restriction\gamma:h\in G\}]=\{h\in G:\supp(h)< \gamma\}\subset G.
	\]
	Thus $j$ extends to an elementary map $\overline{j}$:
	\[
	\begin{array}{llll}
		\overline{j}: & V_\alpha[G_\gamma] &\rightarrow& V_\lambda[G] \\
		&\val_{G_\gamma}(\tau) &\mapsto& \val_G(j(\tau))
	\end{array}
	\]
	Let $M = \overline{j}[V_\alpha[G_\gamma]] \prec V_\lambda[G]$. Since $\overline{j}[H]$ is $M$-generic for $\val_{G}(j(f(\gamma))) = Q$, for any $\mathcal D \in M$, $\mathcal D = \bp{ D_\beta: ~ \beta < \omega_1}$ dense subsets of $Q$, it holds that
	\[
	\begin{array}{lll}
		V_\lambda[G] &\vDash& \forall \alpha < \omega_1  ~ j[H] \cap D_\alpha \neq \emptyset \Rightarrow \\
		V_\lambda[G] &\vDash& \exists K \subset Q \text{ filter } \forall \alpha < \omega_1  ~ K \cap D_\alpha \neq \emptyset \Rightarrow \\
		M &\vDash& \exists K \subset Q \text{ filter } \forall \alpha < \omega_1  ~ K \cap D_\alpha \neq \emptyset\\
	\end{array}
	\]
	Since the latter holds for any $\mathcal D \in M$, it follows that $M \vDash \FA_{\omega_1}(Q)$ and again by elementarity $V_\lambda[G] \vDash \FA_{\omega_1}(Q) \Rightarrow V[G]  \vDash \FA_{\omega_1}(Q)$, concluding the proof.
\end{proof}

%% file: 9_Stationary-sets.tex
\section{Generalized stationary sets}\label{sec:statsets}

In this appendix we recall the properties of generalized stationarity that were 
used throughout these notes in the definition and the analysis of semiproperness. 
Reference texts for this section are \cite{JECH}, \cite[Chapter 2]{LARSON}.

		\begin{definition}
			Let $X$ be an uncountable set. A set $C$ is a \emph{club} on $\pp(X)$ iff there is a function $f_C: ~ X^{<\omega} \rightarrow X$ such that $C$ is the set of elements of $\pp(X)$ closed under $f_C$, i.e.
			\[
			C = \bp{ Y \in \pp(X): ~ f_C[Y]^{<\omega} \subseteq Y }
			\]
			A set $S$ is \emph{stationary} on $\pp(X)$ iff it intersects every club on $\pp(X)$.
		\end{definition}

		\begin{example}
			The set $\bp{X}$ is always stationary since every club contains $X$. Also $\pp(X) \setminus \bp{X}$ and $\qp{X}^\kappa$ are stationary for any $\kappa \leq \vp{X}$ (following the proof of the well-known downwards L\"owhenheim-Skolem Theorem). Notice that every element of a club $C$ must contain $f_C(\emptyset)$, a fixed element of $X$.
		\end{example}

		\begin{remark}
			The reference to the support set $X$ for clubs or stationary sets may be omitted, since every set $S$ can be club or stationary only on $\bigcup S$.
		\end{remark}

		There is one more property of stationary sets that is worth to mention. Given any first-order structure $M$, from the set $M$ we can define a Skolem function $f_M: M^{<\omega} \rightarrow M$ (i.e., a function coding solutions for all existential first-order formulas over $M$). Then the set $C$ of all elementary submodels of $M$ contains a club (the one corresponding to $f_M$). Henceforth, every set $S$ stationary on $X$ must contain an elementary submodel of any first-order structure on $X$.

		\begin{definition}
			A set $S$ is \emph{subset modulo club} of $T$, in symbols $S \subseteq^* T$, iff $\bigcup S = \bigcup T = X$ and there is a club $C$ on $X$ such that $S \cap C \subseteq T \cap C$. Similarly, a set $S$ is \emph{equivalent modulo club} to $T$, in symbols $S =^* T$, iff $S \subseteq^* T \wedge T \subseteq^* S$.
		\end{definition}

		\begin{definition}
			The \emph{club filter} on $X$ is 
			\[
				\CF_X = \bp{C \subset \pp(X): ~ C \text{ contains a club} }.
			\]
			Similarly, the \emph{non-stationary ideal} on $X$ is 
			\[
				\NS_X = \bp{A \subset \pp(X): ~ A \text{ not stationary}}.
			\]
		\end{definition}

		\begin{remark}
			If $\vp{X} = \vp{Y}$, then $\pp(X)$ and $\pp(Y)$ are isomorphic and so are $\CF_X$ and $\CF_Y$ (or $\NS_X$ and $\NS_Y$): then we can suppose $X \in \ON$ or $X \supseteq \omega_1$ if needed.
		\end{remark}

		\begin{lemma}
			$\CF_X$ is a $\sigma$-complete filter on $\pp(X)$, and the stationary sets are exactly the $\CF_X$-positive sets.
		\end{lemma}
		\begin{proof}
			$\CF_X$ is closed under supersets by definition. Given a family of clubs $C_i$, $i < \omega$, let $f_i$ be the function corresponding to the club $C_i$. Let $\pi: \omega \rightarrow \omega^2$ be a surjection, with components $\pi_1$ and $\pi_2$, such that $\pi_2(n) \leq n$. Define $g: X^{<\omega} \rightarrow X$ to be $g(s) = f_{\pi_1(\vp{s})}(s \res \pi_2(\vp{s}))$. It is easy to verify that $C_g = \bigcap_{i < \omega} C_i$.
		\end{proof}

		\begin{definition}
			Given a family $\bp{S_a \subseteq \pp(X): ~ a \in X}$, the \emph{diagonal union} of the family is $\nabla_{a \in X} S_a = \bp{z \in \pp(X): ~ \exists a \in z ~ z \in S_a}$, and the \emph{diagonal intersection} of the family is $\Delta_{a \in X}  S_a = \bp{z \in \pp(X): \forall a \in z ~ z \in S_a}$.
		\end{definition}

		\begin{lemma}[Fodor] \label{sFodor}
			$\CF_X$ is normal, i.e. is closed under diagonal intersection. Equivalently, every function $f: ~ \pp(X) \rightarrow X$ that is regressive on a $\CF_X$-positive set is constant on a $\CF_X$-positive set.
		\end{lemma}
		\begin{proof}
			Given a family $C_a$, $a \in X$ of clubs, with corresponding functions $f_a$, let $g(a^\smallfrown s) = f_a(s)$. It is easy to verify that $C_g = \Delta_{a \in X} C_a$.

			Even though the second part of our thesis is provably equivalent to the first one for any filter $\mathcal F$, we shall opt here for a direct proof. Assume by contradiction that $f: ~ \pp(X) \rightarrow X$ is regressive (i.e., $f(Y) \in Y$) in a $\CF_X$-positive (i.e., stationary) set, and $f^{-1}\qp{a}$ is non-stationary for every $a \in X$. Then, for every $a \in X$ there is a function $g_a: ~ \qp{X}^{<\omega} \rightarrow X$ such that the club $C_{g_a}$ is disjoint from $f^{-1}\qp{a}$. Without loss of generality, suppose that $C_{g_a} \subseteq C_a = \bp{Y \subseteq X: ~ a \in Y}$. As in the first part of the lemma, define $g(a^\smallfrown s) = g_a(s)$. Then for every $Z \in C_g$ and every $a \in Z$, $Z$ is in $C_{g_a}$ hence is not in $f^{-1}\qp{a}$ (i.e., $f(Z) \neq a$). So $f(Z) \notin Z$ for any $Z \in C_g$, hence $C_g$ is a club disjoint with the stationary set in which $f$ is regressive, a contradiction.
		\end{proof}

		\begin{remark}
			The club filter is never $\omega_2$-complete, unlike its well-known counterpart on cardinals. Let $Y \subseteq X$ be such that $\vp{Y} = \omega_1$, and $C_a$ be the club corresponding to $f_a: \qp{X}^{<\omega} \rightarrow \bp{a}$; then $C = \bigcap_{a \in Y} C_a = \bp{Z \subseteq X: ~ Y \subseteq Z}$ is disjoint from the stationary set $\qp{X}^\omega$, hence is not a club.
		\end{remark}

		This generalized notion of club and stationary set is closely related to the well-known one defined for subsets of cardinals.

		\begin{lemma} \label{sClassicalOmg1}
			$C \subseteq \omega_1$ is a club in the classical sense if and only if $C \cup \bp{\omega_1}$ is a club in the generalized sense. $S \subseteq \omega_1$ is stationary in the classical sense if and only if it is stationary in the generalized sense.
		\end{lemma}
		\begin{proof}
			Let $C \subseteq \omega_1 + 1$ be a club in the generalized sense. Then $C$ is closed: given any $\alpha = \sup{\alpha_i}$ with $f[\alpha_i]^{<\omega} \subseteq \alpha_i$, $f[\alpha]^{<\omega} = \bigcup_i f[\alpha_i]^{<\omega} \subseteq \bigcup_i \alpha_i = \alpha$. Furthermore, $C$ is unbounded: given any $\beta_0 < \omega_1$, define a sequence $\beta_i$ by taking $\beta_{i+1} = \sup{f[\beta_i]^{<\omega}}$. Then $\beta_\omega = \sup{\beta_i} \in C$.

			Let now $C \subseteq \omega_1$ be a club in the classical sense. Let $C = \bp{c_\alpha: ~ \alpha < \omega_1}$ be an enumeration of the club. For every $\alpha < \omega_1$, let $\bp{d^\alpha_i: ~ i < \omega} \subseteq c_{\alpha+1}$ be a cofinal sequence in $c_{\alpha+1}$ (eventually constant), and let $\bp{e^\alpha_i: ~ i < \omega} \subseteq \alpha$ be an enumeration of $\alpha$. Define $f_C$ to be $f_C((c_\alpha)^n) = d^\alpha_n$, $f_C(0^\smallfrown \alpha^n) = e^\alpha_n$, and $f_C(s) = 0$ otherwise. The sequence $e^\alpha_i$ forces all closure points of $f_C$ to be ordinals, while the sequence $d^\alpha_i$ forces the ordinal closure points of $f_C$ being in $C$.
		\end{proof}

		\begin{lemma} \label{sClassicalK}
			If $\kappa$ is a cardinal with cofinality at least $\omega_1$, $C \subseteq \kappa$ contains a club in the classical sense if and only if $C \cup \bp{\kappa}$ contains the ordinals of a club in the generalized sense. $S \subseteq \kappa$ is stationary in the classical sense if and only if it is stationary in the generalized sense.
		\end{lemma}
		\begin{proof}
			If $C$ is a club in the generalized sense, then $C \cap \kappa$ is closed and unbounded by the same reasoning of Lemma \ref{sClassicalOmg1}. Let now $C$ be a club in the classical sense, and define $f: ~ \kappa^{< \omega} \rightarrow \kappa$ to be $f(s) = \min \bp{c \in C: \sup{s} < c}$. Then $C_f \cap \kappa$ is exactly the set of ordinals in $C \cup \bp{\kappa}$ that are limits within $C$.
		\end{proof}

		\begin{remark}
			If $S$ is stationary in the generalized sense on $\omega_1$, then $S \cap \omega_1$ is stationary (since $\omega_1+1$ is a club by Lemma \ref{sClassicalOmg1}), while this is not true for $\kappa > \omega_1$. In this case, $\pp(\kappa) \setminus (\kappa+1)$ is a stationary set: given any function $f$, the closure under $f$ of $\bp{\omega_1}$ is countable, hence not an ordinal.
		\end{remark}

		\begin{lemma}[Lifting and Projection] \label{sLifting}
			Let $X \subseteq Y$ be uncountable sets. If $S$ is stationary on $\pp(X)$, then $S \uparrow Y = \bp{B \subseteq Y: ~ B \cap X \in S}$ is stationary. If $S$ is stationary on $\pp(Y)$, then $S \downarrow X = \bp{B \cap X: ~ B \in S}$ is stationary.
		\end{lemma}
		\begin{proof}
			For the first part, given any function $f: ~ \qp{X}^{<\omega} \rightarrow X$, extend it in any way to a function $g: ~ \qp{Y}^{<\omega} \rightarrow Y$. Since $S$ is stationary, there exists a $B \in S$ closed under $g$, hence $B \cap X \in S \downarrow X$ is closed under $f$.

			For the second part, fix an element $x \in X$. Given any function $f: ~ \qp{Y}^{<\omega} \rightarrow Y$, replace it with a function $g: ~ \qp{Y}^{<\omega} \rightarrow Y$ such that for any $A \subset Y$, $g\qp{[A]^{<\omega}}$ contains $A \cup \bp{x}$ and is closed under $f$. To achieve this, fix a surjection $\pi: ~ \omega \rightarrow \omega^2$ (with projections $\pi_1$ and $\pi_2$) such that $\pi_2(n) \leq n$ for all $n$, and an enumeration $\ap{t^n_i: ~ i < \omega}$ of all first-order terms with $n$ variables, function symbols $f_i$ for $i \leq n$ (that represent an $i$-ary application of $f$) and a constant $x$. The function $g$ can now be defined as $g(s) = t^{\pi_2(\vp{s})}_{\pi_1(\vp{s})}(s \res \pi_2(\vp{s}))$. Finally, let $h: ~ \qp{X}^{<\omega} \rightarrow X$ be defined by $h(s) = g(s)$ if $g(s) \in X$, and $h(s) = x$ otherwise. Since $S$ is stationary, there exists a $B \in S$ with $h\qp{[B]^{<\omega}} \subseteq B$, but $h\qp{[B]^{<\omega}} = g\qp{[B]^{<\omega}} \cap X$ (since $x$ is always in $g\qp{[B]^{<\omega}}$) and $g\qp{[B]^{<\omega}} \supset B$, so actually $h\qp{[B]^{<\omega}} = g\qp{[B]^{<\omega}} \cap X = B \in S$. Then, $g\qp{[B]^{<\omega}} \in S \uparrow Y$ and $g\qp{[B]^{<\omega}}$ is closed under $f$ (by definition of $g$).
		\end{proof}

		\begin{remark}
			Following the same proof, a similar result holds for clubs. If $C_f$ is club on $\pp(X)$, then $C_f \uparrow Y = C_g$ where $g = f ~\cup~ \mathrm{Id}_{Y \setminus X}$. If $C_f$ is club on $\pp(Y)$ such that $\bigcap C_f$ intersects $X$ in $x$, and $g, h$ are defined as in the second part of Theorem \ref{sLifting}, $C_f \downarrow X = C_h$ is club. If $\bigcap C_f$ is disjoint from $X$, $C_f \downarrow X$ is not a club, but is still true that it contains a club (namely, $\cp{C_f \cap C_{\bp{x}}} \downarrow X$ for any $x \in X$).
		\end{remark}

		\begin{theorem}[Ulam] \label{sUlam}
			Let $\kappa$ be an infinite cardinal. Then for every stationary set $S \subseteq \kappa^+$, there exists a partition of $S$ into $\kappa^+$ many disjoint stationary sets.
		\end{theorem}
		\begin{proof}
			For every $\beta \in [\kappa,\kappa^+)$, fix a bijection $\pi_\beta: ~ \kappa \rightarrow \beta$. For $\xi < \kappa$, $\alpha < \kappa^+$, define $A^\xi_\alpha = \bp{\beta < \kappa^+: ~ \pi_\beta(\xi) = \alpha}$ (notice that $\beta > \alpha$ when $\alpha \in \ran(\pi_\beta)$). These sets can be fit in a $(\kappa \times \kappa^+)$-matrix, called \emph{Ulam Matrix}, where two sets in the same row or column are always disjoint. Moreover, every row is a partition of $\bigcup_{\alpha < \kappa^+} A^\xi_\alpha = \kappa^+$, and every column is a partition of $\bigcup_{\xi < \kappa} A^\xi_\alpha = \kappa^+ \setminus (\alpha+1)$.

			Let $S$ be a stationary subset of $\kappa^+$. For every $\alpha < \kappa^+$, define $f_\alpha: ~ S \setminus (\alpha+1) \rightarrow \kappa$ by $f_\alpha(\beta) = \xi$ if $\beta \in A^\xi_\alpha$. Since $\kappa^+ \setminus (\alpha+1)$ is a club, every $f_\alpha$ is regressive on a stationary set, then by Fodor's Lemma \ref{sFodor} there exists a $\xi_\alpha < \kappa$ such that $f^{-1}_\alpha\qp{\bp{\xi_\alpha}} = A^{\xi_\alpha}_\alpha \cap S$ is stationary. Define $g: ~ \kappa^+ \rightarrow \kappa$ by $g(\alpha) = \xi_\alpha$, $g$ is regressive on the stationary set $\kappa^+ \setminus \kappa$, again by Fodor's Lemma \ref{sFodor} let $\xi^* < \kappa$ be such that $g^{-1}\qp{\bp{\xi^*}} = T$ is stationary. Then, the row $\xi^*$ of the Ulam Matrix intersects $S$ in a stationary set for stationary many columns $T$. So $S$ can be partitioned into $S \cap A^{\xi^*}_\alpha$ for $\alpha \in T \setminus \bp{\min(T)}$, and $S \setminus \bigcup_{\alpha \in T \setminus \bp{\min(T)}} A^{\xi^*}_\alpha$.
		\end{proof}

		\begin{remark}
			In the proof of Theorem \ref{sUlam} we actually proved something more: the existence of a Ulam Matrix, i.e. a $\kappa \times \kappa^+$-matrix such that every stationary set $S \subseteq \kappa^+$ is compatible (i.e., has stationary intersection) with stationary many elements of a certain row.
		\end{remark}